\documentclass[12pt]{amsart}

\usepackage{fullpage}
\usepackage{amsfonts}
\usepackage{amsthm}
\usepackage{enumerate}
\usepackage{verbatim}
\usepackage{graphicx}

\newcommand{\R}{\mathbb{R}}
\newcommand{\C}{\mathbb{C}}
\newcommand{\N}{\mathbb{N}}

\newcommand{\ve}{\varepsilon}

\DeclareMathOperator{\tr}{tr}

\DeclareMathOperator{\maxroot}{maxroot}
\DeclareMathOperator{\minroot}{minroot}

\newcommand{\E}[2]{\mathbb{E}_{#1}\left[ #2 \right] }
\renewcommand{\P}[2]{\mathbb{P}_{#1}\left( #2 \right)}
\newcommand{\var}[1]{\mathbf{Var}\left[ #1 \right] }

\renewcommand{\Im}{\operatorname{Im}}
\newcommand{\bI}{\mathbf I}

\numberwithin{equation}{section}
\theoremstyle{plain} 

\newtheorem{theorem}{Theorem}[section]
\newtheorem*{theorem*}{Theorem}
\newtheorem{corollary}[theorem]{Corollary}
\newtheorem{lemma}[theorem]{Lemma}
\newtheorem{proposition}[theorem]{Proposition}

\theoremstyle{definition}
\newtheorem{definition}[theorem]{Definition}

\theoremstyle{remark}

\newtheorem{example}[theorem]{Example}

\begin{document}

\title{On Akemann-Weaver Conjecture}

\author{Marcin Bownik}

\address{Department of Mathematics, University of Oregon, Eugene, OR 97403--1222, USA}
\email{mbownik@uoregon.edu}

\date{\today}



\thanks{The author was supported in part by the NSF grant DMS-1956395. The author thanks Chris Phillips and Nikhil Srivastava for useful discussions on the subject of this paper.}

\begin{abstract}
Akemann and Weaver \cite{AW} showed Lyapunov-type theorem for rank one positive semidefinite matrices which is an extension of Weaver's KS$_2$ conjecture \cite{We} that was proven by Marcus, Spielman, and Srivastava \cite{MSS} in their breakthrough solution of the Kadison-Singer problem \cite{KS}. They conjectured   that a similar result holds for higher rank matrices. We prove the conjecture of Akemann and Weaver  by establishing Lyapunov-type theorem for trace class operators. In the process we prove a matrix discrepancy result for sums of hermitian matrices. This extends rank one result of Kyng, Luh, and Song \cite{KLS} who established an improved bound in Lyapunov-type theorem of Akemann and Weaver. 
\end{abstract}

\maketitle

\section{Introduction}

The solution of the Kadison-Singer problem by Marcus, Spielman, and Srivastava \cite{MSS} has resolved in one stroke several outstanding problems in analysis. This is due to the fact that the Kadison-Singer problem \cite{KS} was known to be equivalent to several well-known problems such as Anderson paving conjecture \cite{AA, And, CE, CEKP}, Bourgain-Tzafriri restricted invertibility conjecture \cite{BT1, BT2, BT3}, Feichtinger's conjecture \cite{BS, CCLV}, and Weaver's conjecture \cite{We}. We refer to the surveys \cite{CFTW, CT} and the papers \cite{MB, CT2} discussing the solution of the Kadison-Singer problem and its various ramifications.

Marcus, Spielman, and Srivastava resolved the Kadison-Singer problem by proving 
Weaver's KS$_2$ conjecture.
For any vectors $u_1,\ldots,u_m \in \C^d$ such that
\begin{equation}\label{i1}
\bigg\| \sum_{i=1}^m u_i u_i^* \bigg\| \le 1 \qquad\text{and}\qquad
\max_i ||u_i||^2 \le \ve
\end{equation}
there exists a partition of $[m]$ into sets $I_1$ and $I_2$ such that
\[
\bigg\| \frac{1}{2} \sum_{i=1}^m u_i u_i^* - \sum_{i\in I_k} u_i u_i^* \bigg\| \le O(\sqrt{\ve}) 
\qquad k=1,2.
\]
Based on this development Akemann and Weaver \cite{AW} have shown an interesting variant of classical Lyapunov's theorem on ranges of non-atomic finitely dimensional vector-valued measures in the setting of discrete frames. A similar result in the setting of continuous non-atomic frames, which was shown by the author \cite{MB2, MB3}, does not rely on the solution of the Kadison-Singer problem. Akemann and Weaver's result is a generalization of KS$_2$ conjecture for vectors satisfying \eqref{i1}.
For any $0 \le t_i \le 1$, $i\in [m]$, there exists a set $I_0 \subset [m]$ such that
\begin{equation}\label{i3}
\bigg\| \sum_{i=1}^m t_i u_i u_i^* - \sum_{i\in I_0} u_i u_i^* \bigg\| \le O(\ve^{1/8}).
\end{equation}
Recently, Kyng, Luh, and Song \cite{KLS} have greatly improved the $\epsilon$ dependence and provided a small explicit constant in \eqref{i3} by showing that $O(\ve^{1/8})$  can be replaced by $2 \sqrt{\ve}$. This result was further refined in \cite{XXZ1} by showing that \eqref{i3} holds with a smaller multiplicative constant $\tfrac 32 \sqrt{\ve}$. 

Akemann and Weaver \cite{AW} have also conjectured a generalization of their Lyapunov-type theorem for higher rank positive matrices. We show the validity of their conjecture not only for matrices in finitely dimensional spaces, but also for trace class operators in a separable Hilbert space. Let $\{T_i\}_{i\in I}$ be any family of positive trace class operators such that
\begin{equation}\label{i4}
\sum_{i\in I} T_i \le \bI \quad\text{and}\quad \tr(T_i)\le \epsilon \qquad\text{for all }i \in I.
\end{equation}
Then for any $0\le t_i \le 1$, $i\in I$, there exists a subset $I_0 \subset I$ such that
\begin{equation}\label{i5}
\bigg\| \sum_{i\in I} t_i T_i  - \sum_{i\in I_0} T_i  \bigg\| \le 2 \sqrt{\epsilon}.
\end{equation}
When all operators $T_i$ have rank $1$, our result recovers Lyapunov-type theorem of Akemann and Weaver \cite{AW} with the explicit bound shown Kyng, Luh, and Song \cite[Corollary 1.8]{KLS}.

To establish our main result we show a matrix discrepancy bound. Weaver's KS$_2$ result can be equivalently stated as a discrepancy statement for sums of hermitian matrices. Under the assumption \eqref{i1}, there exists signs $\ve_i = \pm 1$ such that
\begin{equation}\label{i6}
\bigg\| \sum_{i=1}^m \ve_i u_i u_i^* \bigg\| \le O(\sqrt{\ve}).
\end{equation}
The main result of Kyng, Luh, and Song \cite[Theorem 1.4]{KLS} yields the following extension of \eqref{i6} for biased choice of signs. 

\begin{theorem}\label{klsone}
Let $\xi_1,\ldots, \xi_m$ be any independent scalar random variables with finite support. For $u_1,\ldots,u_m \in \C^m$ define
\[
\sigma^2= \bigg\| \sum_{i=1}^m \var{\xi_i}(u_i u_i^*)^2 \bigg\|.
\]
Then there is an outcome such that
\[
\bigg\| \sum_{i=1}^m \E{}{\xi_i} u_i u_i^* - \sum_{i=1}^m \xi_i u_i u_i^* \bigg\| \le 4 \sigma.
\]
\end{theorem}
We show that the rank one assumption in the above result can be dropped. Moreover, the requirement that matrices are positive semidefinite can be relaxed by hermitian assumption albeit with a slightly worse constant, see Corollary \ref{klsh}.

\begin{theorem}\label{kls}
Suppose that $\xi_1,\ldots, \xi_m$ are jointly independent scalar random variables which take finitely many values.  Let $A_1, \ldots, A_m$ be $d \times d$ positive semidefinite matrices. Define 
\begin{equation}\label{kls1}
\sigma^2= \max \bigg( \max_{i=1,\ldots,m} \var{\xi_i} \tr(A_i)^2, \bigg\| \sum_{i=1}^m \var{\xi_i} \tr(A_i) A_i \bigg\| \bigg).
\end{equation}
Then, there is an outcome such that
\begin{equation}\label{kls2}
\bigg\| \sum_{i=1}^m (\xi_i-\E{}{\xi_i}) A_i \bigg\| \le 4 \sigma.
\end{equation}
\end{theorem}

To see that Theorem \ref{kls} implies Theorem \ref{klsone}, suppose that all matrices $A_i$ have rank $1$. Then we have $\tr(A_i)=||A_i||$, $\tr(A_i)A_i =A_i^2$, and the formula \eqref{kls1} reduces to
\[
\sigma^2= \ \bigg\| \sum_{i=1}^m \var{\xi_i} A_i^2 \bigg\|.
\]
Thus, Theorem \ref{kls} is a generalization of rank one result Theorem \ref{klsone}. 

As observed in \cite[Remark 1.5]{KLS}, the full strength of Theorem \ref{kls} is achieved for random variables $\xi_i$ taking only two values. Indeed, one can replace random variables $\xi_i$ by alternative random variables $\xi_i'$ that have identical means and take only two values. By shifting the probability mass, we can arrange that $\xi_i'$ takes the nearest value of $\xi_i$ above and below the mean $\E{}{\xi_i}=\E{}{\xi'_i}$, while reducing the variance. This also reduces the corresponding value $\sigma^2$ in \eqref{kls1} yielding a better estimate \eqref{kls2}.

Theorem \ref{kls} can be contrasted with the well-known bounds for sums of hermitian matrices. Let $\xi_i$, $i\in [m]$ are independent Rademacher random variables (taking values $\pm 1$ with probability $\tfrac 12$). Let $A_1, \ldots, A_m$ be $d \times d$ hermitian matrices. Theorem due to Oliveira \cite[Theorem 1.2]{Tro} states that for $t>0$,
\[
\P{}{\bigg\| \sum_{i=1}^m \xi_i A_i \bigg\| \ge t \sigma } \le 2d e^{-t^2/2},
\qquad\text{where }  \sigma^2=\bigg\| \sum_{i=1}^m  A_i^2 \bigg\|.
\]
Hence, with high probability
\[
\bigg\| \sum_{i=1}^m \xi_i A_i \bigg\| = O(\sigma\sqrt{\log d}) .
\]
This implies a bound of $O(\sigma\sqrt{\log d})$ on the matrix discrepancy. In general, one can not remove the dimensional factor $\sqrt{\log d}$, see \cite[Proposition 1.7]{XXZ1}. In fact, even if matrices $A_1,\ldots, A_m$ are diagonal with non-negative entries 
$\sum_{i=1}^m A_i = \mathbf I$ and their norms are small $\max_i ||A_i|| \le \ve$, then the matrix discrepancy could be larger than $ O(\sigma)$, where
\[
\sigma^2 = \bigg\| \sum_{i=1}^m  A_i^2 \bigg\|  \le \bigg\| \sum_{i=1}^m  ||A_i||A_i \bigg\|  \le\ve.
\]
This is due to the result of Akemann and Anderson \cite{AA2} on the optimality of continuous Beck-Fiala theorem. Instead, Akemann and Weaver \cite{AW} conjectured that the discrepancy is at most $O(\sqrt{\ve})$ for positive matrices with small traces $\max_i \tr(A_i) \le \ve$. This is a consequence of Theorem \ref{kls}, which imbeds traces of matrices in the definition of $\sigma$, thus getting rid of dimensional dependence on $d$.

The proof of our main result requires developing some new properties of mixed characteristic polynomials and barrier function estimates which were originally introduced in \cite{MSS}. In Section \ref{S2} we show that a mixed characteristic polynomial $\mu[A_1,\ldots, A_m]$ is real-rooted if matrices $A_i$ are either positive or negative semidefinite despite that real-rootedness usually fails for hermitian matrices. We also show the monotonicity of maximum root of $\mu[A_1,\ldots, A_m]$ with respect to partial order $\le$ on hermitian matrices. As a consequence we establish the key bound on sums of positive definite matrices with signs $\ve_i=\pm1$.
\begin{equation}\label{i10}
\bigg\| \sum_{i=1}^m \ve_i A_i \bigg\| \le \maxroot( \mu[\ve_1 A_1,\ldots, \ve_m A_m] \cdot  \mu[-\ve_1 A_1,\ldots, -\ve_m A_m]).
\end{equation}
This enables us to simultaneously control both the the largest and the smallest eigenvalues of the sum  $\sum_{i=1}^m \ve_i A_i$ by a single polynomial. More importantly, using stability preserving results of Borcea and Br\"and\'en \cite{BB}, we show that product polynomials in \eqref{i10} form an interlacing family of polynomials \cite{MSS0, MSS}. Consequently, the proof of Theorem \ref{kls} reduces to the bound on the largest root of the expected polynomial
\[
\E{}{\mu[\xi_1 A_1,\ldots, \xi_m A_m] \cdot  \mu[-\xi_1 A_1,\ldots, -\xi_m A_m]}.
\]
In turn, this involves a quadratic variant of a mixed characteristic polynomial. In Section \ref{S3} we establish the bound on the largest root of a quadratic mixed characteristic polynomial which is analogous to \cite[Theorem 5.1]{MSS}. This requires multivariate barrier argument that is adapted to the stability preserving second order differential operator $1-\partial_z\partial_w$ rather than first order operator $1-\partial_z$ employed in \cite{MSS}. We use an indirect approach using the barrier method for the second order operator $1-\tfrac 12 \partial^2_z$, which was originally introduced by Anari and Oveis Gharan \cite{AG} and employed by Kyng, Luh, and Song \cite{KLS}. Finally, in Section \ref{S4} we show Lyapunov-type theorem for positive trace class operators, thus confirming the conjecture of Akemann and Weaver.

\section{Properties of mixed characteristic polynomials}\label{S2}

The main goal of this section is to show several results involving a mixed characteristic polynomial, which was introduced by Marcus, Spielman, and Srivastava \cite{MSS} in their solution of the Kadison-Singer problem. We also
show a strengthening of their main probabilistic result on random positive semidefinite matrices. The only difference between Theorem \ref{thmp} and \cite[Theorem 1.4]{MSS} is that we drop the rank one assumption. This result was independently announced by Michael Cohen \cite{Cohen} at the Banff workshop on ``Algebraic and Spectral Graph Theory", who prematurely passed away in 2017. An extension of Theorem \ref{thmp} in the realm of hyperbolic polynomials was shown by Br\"and\'en \cite[Theorem 6.1]{Bra} with more precise bounds depending on the expected ranks of random matrices $X_i$. Our auxiliary goal is to present a simple proof of Theorem \ref{thmp}, which requires a minimal number of modifications to the original proof in \cite{MSS}.

\begin{theorem}\label{thmp}
Let $\epsilon > 0$. Suppose that $X_1, \dots, X_m$ are jointly independent $d\times d$ positive semidefinite random matrices, which take finitely many values and satisfy  
\begin{equation}\label{thmp1}
\sum_{i=1}^m \E{}{X_{i} } \le \bI
\qquad\text{and}\qquad
\E{}{ \tr X_i } \leq \epsilon \quad\text{for all } i.
\end{equation}
Then,
\begin{equation}\label{thmp2}
\P{}{\bigg\| \sum_{i=1}^m X_i \bigg\| \leq (1 + \sqrt{\epsilon})^2} > 0.
\end{equation}
\end{theorem}

\subsection{Mixed characteristic polynomial} As in \cite{MSS}, the proof of Theorem \ref{thmp} involves the concept of a mixed characteristic polynomial.

\begin{definition} \label{dmcp}
Let $A_1, \ldots, A_m$ be $d\times d$ matrices. The mixed characteristic polynomial is defined for $z\in \C$ by
\[
\mu[A_1,\ldots,A_m](z) =  \bigg(\prod_{i=1}^m (1 - \partial_{z_i}) \bigg) \det \bigg( z \mathbf I + \sum_{i=1}^m z_i A_i \bigg)\bigg|_{z_1=\ldots=z_m=0}.
\]
\end{definition}

By determinant expansion one can show that $\det \bigg( z \mathbf I + \sum_{i=1}^m z_i A_i \bigg)$ is  a polynomial in $\C[z,z_1,\ldots,z_m]$ of degree $\le d$ in each variable. Hence, $\mu[A_1,\ldots,A_m](z)$, like characteristic polynomial, is a monic polynomial in $\C[z]$ of degree $d$. We need to employ a few basic properties of $\mu$. The following lemma is implicitly stated in Tao's exposition \cite{Tao} and it is shown explicitly in the author's survey on the Kadison-Singer problem \cite[Lemma 3.2]{MB}.

\begin{lemma}\label{mas} For a fixed $z\in \C$, the mixed characteristic polynomial mapping 
\[
\mu: M_{d\times d}(\C) \times \ldots\times  M_{d \times d}(\C) \to \C
\]
 is multi-affine and symmetric. That is, $\mu$ affine in each variable and its value is the same  for any permutation of its arguments $A_1,\ldots,A_m$.
\end{lemma}

\begin{proof}
First we shall show that for any $d\times d$ matrix $B$, a function 
\[
f: M_{d\times d}(\C) \to \C, \qquad f(A_1) = (1 - \partial_{z_1}) \det( B+ z_1 A_1)|_{z_1=0} \qquad\text{for } A_1 \in M_{d\times d}(\C)
\]
is affine. Indeed, if $B$ is invertible, then by Jacobi's formula
\[
f(A_1)= \det(B)  - \det( B) \partial_{z_1}\det(\mathbf I + z_1 B^{-1}A_1)|_{z_1=0} = \det(B)(1 - \tr(B^{-1}A_1)).
\]
Since invertible matrices are dense in the set of all matrices, by continuity we deduce the general case. Therefore, for any choice of matrices $A_2, \ldots, A_m$, a mapping
\[
(M_{d\times d}(\C),\C^{m-1}) \ni (A_1, z_2,\ldots,z_m) \mapsto  (1 - \partial_{z_1})\det \bigg( z \mathbf I + \sum_{i=1}^m z_i A_i \bigg) \bigg|_{z_1=0}
\]
is affine in the $A_1$ variable and a polynomial of degree $\le d$ in $z_2,\ldots,z_m$ variables. Applying linear operators, such as partial differential operators with constant coefficients $(1-\partial_{z_i})$, $i=2,\ldots,m$, preserves this property. Consequently, the mapping $A_1 \mapsto \mu[A_1,\ldots,A_m](z)$ is affine. Since $\mu$ is symmetric with respect to any permutation of its arguments $A_1, \ldots, A_m$, $\mu$ is multi-affine. 
\end{proof}

Lemma \ref{mas} implies the following random multilinearization formula \cite[Lemma 3.4]{MB}. Note that the rank one assumption, for example in \cite[Corollary 4]{Tao}, is not needed.  

\begin{lemma}\label{mux}
Let $X_1,\ldots ,X_m$ be $d\times d$ jointly independent random matrices, which take finitely many values. Then,
\begin{equation}\label{mux1}
\E{}{\mu[X_1,\ldots, X_m](z)} = \mu[\E{}{X_1}, \ldots, \E{}{X_m}](z) \qquad z\in\C.
\end{equation}
\end{lemma}

To formulate more properties of mixed characteristic polynomial we need the concept of a real stable polynomial.

\begin{definition}
Let $\C_+ = \{ z \in \C : \Im(z) > 0 \}$ be the upper half plane. 
We say that a polynomial $p \in \C[z_1, \ldots, z_m]$ is {\em stable} if $p(z_1,\ldots,z_m) \neq 0$ for every $(z_1,\ldots,z_m) \in \C_+^m$.
A polynomial is called {\em real stable} if it is stable and all of its coefficients are real.
\end{definition}

Note that a univariate polynomial is real stable if and only if it is real-rooted, i.e., all its roots are real. It is easy to show that the following transformations on $\C[z_1, \ldots, z_m]$ preserve stability:
\begin{enumerate}
\item
(permutation) 
$p(z_1,\ldots,z_m) \mapsto p(z_{\sigma(1)},\ldots,z_{\sigma(m)})$ for some $\sigma \in S_m$,
\item 
(scaling) 
$p(z_1,\ldots,z_m) \mapsto p(t z_1,\ldots,z_m)$ for some $t>0$,
\item
(restriction)
$p(z_1,\ldots,z_m) \mapsto p(t,\ldots,z_m)$ for some $t\in \R$, 
\item
(diagonalization)
$p(z_1,z_2,\ldots,z_m) \mapsto p(z_2,z_2,z_3,\ldots,z_m)$.
\end{enumerate}

Borcea and Br\"and\'en \cite[Theorem 2.1]{BB09} have characterized linear transformations that preserve stability.

\begin{theorem}\label{bb09}
For  $k \in \N$, let $\C_k[z_1,\ldots,z_m]$ denote the space space of complex polynomials of degree at most $k$ in each variable $z_i$, $i=1,\ldots,m$. Let 
\[
T: \C_k[z_1,\ldots,z_m] \to \C[z_1, \ldots, z_m]
\]
be a linear operator such that $\dim(\operatorname{range}(T))\ge 2$. Define the symbol of $T$ as a polynomial $G_T \in \C[z_1,\ldots,z_m,w_1,\ldots,w_m]$ given by 
\[
G_T(z_1,\ldots,z_m,w_1,\ldots,w_m)= T((z_1+w_1)^k \ldots (z_m+w_m)^k),
\]
where the action of $T$ is extended to variables $w_1,\ldots,w_m$ by treating them as constants.

Then, the operator $T$ preserves stability, i.e., $T(p)$ is a stable polynomial for any stable $p\in \C_k[z_1, \ldots, z_m]$, if and only, the symbol $G_T$ is a stable polynomial in $2m$ variables.
\end{theorem}

Borcea and Br\"and\'en \cite[Theorem 1.2]{BB} have also shown a similar characterization for the special case of partial differential operators with polynomial coefficients.

\begin{theorem}\label{bb10}
Let $T: \C[z_1,\ldots,z_m] \to \C[z_1, \ldots, z_m]$
be a linear operator of the form
\[
T = \sum_{\alpha, \beta \in \N_0^m} a_{\alpha,\beta} z^\alpha \partial^\beta,
\]
where $z=(z_1,\ldots,z_m)$ and coefficients $a_{\alpha,\beta}\in \C$ are nonzero only for a finite number of pairs of multi-indices $(\alpha,\beta)$. Define the symbol of $T$ as a polynomial $F_T \in \C[z_1,\ldots,z_m,w_1,\ldots,w_m]$ given by 
\[
F_T(z_1,\ldots,z_m,w_1,\ldots,w_m)= \sum_{\alpha, \beta \in \N_0^m} a_{\alpha,\beta} z^\alpha (-w)^\beta
\]
where $w=(w_1,\ldots,w_m)$. 

Then, the operator $T$ preserves stability, i.e., $T(p)$ is a stable polynomial for any stable $p\in \C[z_1, \ldots, z_m]$, if and only, the symbol $F_T$ is a stable polynomial in $2m$ variables.
\end{theorem}

As an immediate corollary of Theorem \ref{bb10} we show stability preservation of the following two differential operators, which will be used subsequently.

\begin{lemma} \label{spd}
The following partial differentiation operators on $\C[z_1, \ldots, z_m]$ preserve stability:
\begin{enumerate}
\item {\rm (differentiation operator of order 1)}
$1+t\partial_{z_1}$, where $t\in \R$,
\item {\rm (mixed differentiation operator)}
$1+t\partial_{z_1}-t\partial_{z_2}-s^2\partial_{z_1}\partial_{z_2}$, where $t,s\in \R$ and $t^2\le s^2$.
\end{enumerate}
\end{lemma}

\begin{proof}
Part (i) can be proven by a direct argument, see \cite[Corollary 3.8]{MSS} or \cite[Lemma 3.6]{MB}. 
Alternatively, it follows immediately from Theorem \ref{bb10}. The symbol of the operator $1+t\partial_{z_1}$ is a polynomial $1-tw_1$, which is real stable. 

Likewise, to prove (ii) it suffices to show that the symbol $F_T$ of the operator $T=1+t\partial_{z_1}-t\partial_{z_2}-s^2\partial_{z_1}\partial_{z_2}$ is real stable. That is, we need to show that the polynomial $F_T(w_1,w_2)=1-tw_1+tw_2-s^2w_1w_2 \in \R[w_1,w_2]$ is stable. This follows from the characterization of real stable multi-affine polynomials due to Branden \cite[Theorem 5.6]{Bra07}. Alternatively, it can be shown by the following elementary argument. 

Suppose that $1-tw_1+tw_2-s^2w_1w_2=0$ for some $w_1 \in \C_+$. Solving for $w_2$ yields
\[
w_2=\frac{tw_1-1}{t-s^2 w_1}= \frac{t^2 w_1+s^2 \overline{w_1} -t-ts^2 |w_1|^2}{|t-s^2w_1|^2}.
\]
Since $t^2 \le s^2$, we deduce that $\Im(w_2)<0$. Hence, the polynomial $F_T$ is stable. By Theorem \ref{bb10}, the operator $T$ preserves stability.
\end{proof}

A non-example of stability preserving operator is $1+\partial_{z_1}\partial_{z_2}$, whose symbol $1+w_1w_2$ is not a stable polynomial.
Next we show a useful extension of a result of Marcus, Spielman, and Srivastava, see \cite[Corollary 4.4]{MSS} on real-rootedness of mixed characteristic polynomial.

\begin{lemma}\label{crs}
If $\ve_1,\ldots,\ve_m \in \R$ and $A_1,\ldots,A_m$ are positive semidefinite hermitian $d\times d$ matrices, then the mixed characteristic polynomial $\mu[\ve_1 A_1,\ldots,\ve_m A_m]$ is real-rooted and monic of degree $d$.
\end{lemma}

\begin{proof}
Note that for any polynomial $p \in \C[z_1,z_2,\ldots,z_m]$ we have
\[
\partial_{z_1} p(\epsilon_1 z_1,z_2,\ldots z_m)\bigg|_{z_1=0} = \epsilon_1 \partial_{z_1} p( z_1,z_2,\ldots z_m)\bigg|_{z_1=0} \in \C[z_2,\ldots,z_m].
\]
Hence, we observe that 
\begin{equation}\label{crs2}
\begin{aligned}
\mu[\ve_1 A_1,\ldots,\ve_m A_m](z) & =  \bigg(\prod_{i=1}^m (1 - \partial_{z_i}) \bigg) \det \bigg( x \mathbf I + \sum_{i=1}^m z_i \ve_i A_i \bigg)\bigg|_{z_1=\ldots=z_m=0}
\\
& =  \bigg(\prod_{i=1}^m (1 -\ve_i \partial_{z_i}) \bigg) \det \bigg( x \mathbf I + \sum_{i=1}^m z_i A_i \bigg)\bigg|_{z_1=\ldots=z_m=0}.
\end{aligned}
\end{equation}
By the result of Borcea and Br\"and\'en \cite[Proposition 1.12]{BB} the polynomial
\[
\det \bigg( x \mathbf I + \sum_{i=1}^m z_i A_i \bigg) \in \R[x,z_1,\ldots,z_m]
\]
is real stable. The operators $(1 -\ve_i \partial_{z_i})$ preserve real stability of polynomials by Lemma \ref{spd}. Setting all variables $z_i$ to zero also preserves the real stability. The resulting polynomial is univariate, and hence it is real-rooted. 
\end{proof}

It is well known that the (usual) characteristic polynomial of a hermitian matrix is real-rooted. In light of Lemma \ref{crs}, one might think that a mixed characteristic polynomial $\mu[A_1,\ldots, A_m]$ is also real-rooted for hermitian matrices $A_1,\ldots,A_m$. However, the following simple example shows that this is not true in general.

\begin{example}\label{hermit}
 Let $A_1= \begin{bmatrix} 1 & 0 \\ 0 & -1 \end{bmatrix}$ and $A_2= \begin{bmatrix} -1 & 0 \\ 0 & 1 \end{bmatrix}$. Then the determinantal polynomial $Q(z,z_1,z_2) = \det ( z \mathbf I + z_1 A_1 + z_2 A_2) = z^2 - (z_1-z_2)^2$. Hence, the mixed characteristic polynomial
\[
\mu[A_1,A_2](z) = (1-\partial_{z_1})(1-\partial_{z_2}) Q |_{z_1=z_2=0} = z^2+2
\]
is not real-rooted.
\end{example}

\subsection{Interlacing family of polynomials}
A key contribution of Marcus, Spielman, and Srivastava is the concept of interlacing family of polynomials \cite{MSS0, MSS}. This method relies on the following well-known lemma about real-rooted polynomials. For the proof see \cite[Lemma 3.9]{MB}.

\begin{lemma}\label{ell}
Let $p_1,\ldots, p_n \in \R[x]$ be real-rooted monic polynomials of the same degree. Suppose that every convex combination \[
\sum_{i=1}^n t_i p_i,\qquad{where } \ \sum_{i=1}^n t_i =1,\ t_i\ge 0
\]
is  a real-rooted polynomial. Then, for any such convex combination there exist $1\le i_0, j_0 \le n$ such that
\begin{equation}\label{ell1}
\maxroot(p_{i_0}) \le \maxroot\bigg(\sum_{i=1}^n t_i p_i \bigg)
\le \maxroot(p_{j_0}).
\end{equation}
\end{lemma}

A stronger variant of Lemma \ref{ell} asserts the existence of a common interlacing polynomial $q$ for the family $p_1,\ldots, p_n $. That is, $q$ a real rooted polynomial of one degree less than the degree of $p_i$'s such that roots of $q$ interlace roots of $p_i$ for every $i=1,\ldots,n$. This fact  justifies the name for interlacing family of polynomials although we will not employ roots other the largest root given by Lemma \ref{ell}.
We shall adopt the following simplified definition for interlacing family of polynomials that we find more convenient to use. It was formulated by Br\"and\'en \cite[Definition 2.1]{Bra}, but it is essentially equivalent to \cite[Definition 3.3]{MSS}. 

\begin{definition}\label{ifp} Let $S_1,\ldots,S_m$ be finite sets and $d\in \N$. For each choice $(s_1,\ldots,s_m) \in S_1 \times \ldots \times S_m$, let $f_{s_1,\ldots,s_m} \in  \R[x]$ be a monic real rooted polynomial of degree $d$. We say that $\{f_{s_1,\ldots,s_m} \}$ is an interlacing family of polynomials if for every choice of independent random variables $\xi_1,\ldots, \xi_m$ taking values in sets $S_1,\ldots,S_m$, respectively, the polynomial $\E{}{f_{\xi_1,\ldots,\xi_m}}$ is real-rooted.
\end{definition}

Once random variables $\xi_1,\ldots, \xi_m$ taking values in sets $S_1,\ldots,S_m$ are fixed, we can define a rooted tree of interlacing polynomials. The root of the tree is $f_\emptyset = \E{}{f_{\xi_1,\ldots,\xi_m}}$. For a partial choice $(s_1,\ldots, s_k) \in S_1 \times \ldots \times S_k$ with $k<m$, we define
\[
f_{s_1,\ldots,s_k} = \E{}{f_{s_1,\ldots, s_k, \xi_{k+1},\ldots,\xi_m}}.
\]
The definition of interlacing family in \cite{MSS} requires merely that $\E{}{f_{s_1,\ldots,s_k,\xi}}$ is real-rooted for any partial choice $s_1,\ldots, s_k$ with $k<m$ and any random variable $\xi$ taking values in $S_{k+1}$. Hence, the definition in \cite{MSS} is formally more general than Definition \ref{ifp}, although this is inconsequential for our considerations. Finally, the hypothesis that all polynomials are monic is not essential since it can be replaced by the assumption that they all have a positive leading coefficient. Then, we have the following result \cite[Theorem 2.3]{Bra}, which is a convenient variant of \cite[Theorem 3.6]{MSS}. We include the proof for convenience.

\begin{theorem}\label{maxint} Let $\{f_{s_1,\ldots,s_m} \}_{(s_1,\ldots,s_m) \in S_1 \times \ldots \times S_m}$ be an interlacing family of polynomials. Let $\xi_1,\ldots, \xi_m$ be independent random variables  taking values in sets $S_1,\ldots,S_m$, respectively. Then, with positive probability
\begin{equation}\label{maxint1}
\maxroot f_{\xi_1,\ldots,\xi_m} \le \maxroot \E{}{f_{\xi_1,\ldots,\xi_m}}.
\end{equation}
In other words, there exists an outcome $(s_1,\ldots,s_m) \in S_1 \times \ldots \times S_m$ such that $\P{}{\xi_i=s_i}>0$ for all $i=1,\ldots,m$, and
\[
\maxroot f_{s_1,\ldots,s_m} \le \maxroot \E{}{f_{\xi_1,\ldots,\xi_m}}.
\]
Likewise, with positive probability
\begin{equation}\label{maxint2}
\maxroot f_{\xi_1,\ldots,\xi_m} \ge \maxroot \E{}{f_{\xi_1,\ldots,\xi_m}}.
\end{equation}
\end{theorem}

\begin{proof} We proceed by induction over $m$. The case when $m=1$ follows from Lemma \ref{ell}. Next assume that $m>1$. Let $\xi$ be any random variable, which is independent of $\xi_1,\ldots,\xi_{m-1}$, taking values in the set $S_m=\{c_1,\ldots,c_n\}$ with probabilities $t_1,\ldots,t_n$, respectively. By Definition \ref{ifp}, the polynomial
\[
\E{}{f_{\xi_1,\ldots,\xi_{m-1},\xi}} = \sum_{i=1}^n t_i \E{}{f_{\xi_1,\ldots,\xi_{m-1}, c_i}}
\]
is real-rooted. Hence, by Lemma \ref{ell} there exists $1 \le i \le n$ such that 
\[
\maxroot (\E{}{f_{\xi_1,\ldots,\xi_{m-1}, c_i}} )
\le
\maxroot(  \E{}{f_{\xi_1,\ldots,\xi_{m-1},\xi_m}} ).
\]
Applying the inductive hypothesis to the interlacing family of polynomials 
\[
\{f_{s_1,\ldots,s_{m-1},c_i} \}_{(s_1,\ldots,s_{m-1}) \in S_1 \times \ldots \times S_{m-1}}
\]
completes the induction step. This proves \eqref{maxint1}. \eqref{maxint2} is shown in the same way.
\end{proof}

There are two basic examples of interlacing families of polynomials which can be defined using a mixed characteristic polynomial.

\begin{proposition}\label{ifpp}
The following are two examples of interlacing families of polynomials.
\begin{enumerate}
\item
Fix $\ve_1, \ldots, \ve_m \in \R$.
 For a finite subsets $S_1,\ldots, S_m$ of $d\times d$ positive semidefinite matrices consider the family
\[
f_{A_1,\ldots,A_m}(x) = \mu[\ve_1 A_1,\ldots,\ve_m A_m](x),
\qquad \text{where } (A_1,\ldots, A_m) \in S_1 \times \ldots \times S_m.
\]
\item Fix positive semidefinite $d\times d$ matrices $A_1,\ldots,A_m$. For finite subsets $S'_1,\ldots, S'_m \subset \R$, consider the family
\[
g_{\ve_1,\ldots,\ve_m}(x) = \mu[\ve_1 A_1,\ldots,\ve_m A_m](x),
\qquad \text{where } (\ve_1,\ldots, \ve_m) \in S'_1 \times \ldots \times S'_m.
\]
\end{enumerate}
\end{proposition}

\begin{proof}
Let $X_1,\ldots ,X_m$ be independent positive semidefinite $d\times d$ random matrices, which take finitely many values.
By Lemma \ref{mux} 
\[
\E{}{\mu[\ve_1 X_1,\ldots, \ve_m X_m](x)} = \mu[\ve_1 \E{}{X_1}, \ldots, \ve_m \E{}{X_m}](x).
\]
Hence, by Lemma \ref{crs}
the family of monic polynomials $\{f_{A_1,\ldots,A_m}\}$ is interlacing. Likewise, let $\xi_1,\ldots,\xi_m$ be independent real random variables, which take finitely many values. 
By Lemma \ref{mux} 
\[
\E{}{\mu[\xi_1 A_1,\ldots, \xi_m A_m](x)} = \mu[ \E{}{\xi_1}A_1, \ldots,  \E{}{\xi_m}A_m ](x).
\]
Hence, by Lemma \ref{crs} the family $\{g_{\ve_1,\ldots,\ve_m}\}$ is also interlacing. 
\end{proof}

In light of Proposition \ref{ifpp} it might be tempting to consider a more general family of polynomials $\mu[A_1,\ldots,A_m](x)$, where matrices $A_1,\ldots,A_m$ are either positive semidefinite or negative semidefinite. However, Example \ref{hermit} shows that this would not be an interlacing family. Nevertheless, the following theorem gives another example of an interlacing family of polynomials, which will show its importance in Section \ref{S3}.

\begin{theorem}\label{qint} Let $A_1,\ldots,A_m$ be positive semidefinite $d\times d$ matrices. Let $S_1,\ldots,S_m$ be finite subsets of $\R$. For $(\ve_1,\ldots, \ve_m) \in S_1 \times \ldots \times S_m$ define a polynomial
\begin{equation}\label{qint0}
f_{\ve_1,\ldots,\ve_m}(x) = \mu[\ve_1 A_1,\ldots,\ve_m A_m](x)\mu[-\ve_1 A_1,\ldots,-\ve_m A_m](x) \in \R[x].
\end{equation}
Then, $\{f_{\ve_1,\ldots,\ve_m} \}_{(\ve_1,\ldots,\ve_m) \in S_1 \times \ldots \times S_m}$ forms an interlacing family of polynomials.
\end{theorem}

\begin{proof}
Using Lemma \ref{crs} and the formula \eqref{crs2} we deduce that
\begin{equation}\label{qint1}
f_{\ve_1,\ldots,\ve_m}(x) =  
\bigg(\prod_{i=1}^m (1 - \ve_i\partial_{z_i}) (1 + \ve_i \partial_{w_i}) \bigg)
 \det \bigg( x \mathbf I + \sum_{i=1}^m z_i  A_i\bigg)
 \det \bigg( x \mathbf I + \sum_{i=1}^m w_i  A_i \bigg)
  \bigg|_{\genfrac{}{}{0pt}{2}{z_1=\ldots=z_m=0}{w_1=\ldots=w_m=0}}.
\end{equation}
is a real-rooted and monic polynomial of degree $2d$. 
Let $\xi_1,\ldots, \xi_m$ be independent random variables taking values in sets $S_1,\ldots,S_m$, respectively.
Then, we compute the expected value of a random partial differential operator
\begin{equation}\label{qint2}
\begin{aligned}
\E{}{\prod_{i=1}^m (1 - \xi_i \partial_{z_i}) (1 + \xi_i \partial_{w_i}) }
& = \prod_{i=1}^m \E{}{ 1 - \xi_i \partial_{z_i} + \xi_i \partial_{w_i}- \xi_i^2\partial_{z_i}\partial_{w_i} }
\\
&=  \prod_{i=1}^m (1 - \E{}{\xi_i}\partial_{z_i} + \E{}{\xi_i} \partial_{w_i}- \E{}{\xi_i^2}\partial_{z_i}\partial_{w_i}).
\end{aligned}
\end{equation}
Since $\E{}{\xi_i}^2 \le \E{}{\xi_i^2}$, by Lemma \ref{spd} we deduce that the above operator preserves stability. Recall that the polynomial
\[
\det \bigg( x \mathbf I + \sum_{i=1}^m z_i  A_i\bigg)
 \det \bigg( x \mathbf I + \sum_{i=1}^m w_i  A_i \bigg) \in \R[x,z_1,\ldots,z_m,w_1,\ldots,w_m],
 \]
is stable. Hence, $\E{}{f_{\xi_1,\ldots,\xi_m}}$ is a real-rooted polynomial as a restriction of a multivariate polynomial to $z_1=\ldots=z_m=w_1=\ldots=w_m=0$.
\end{proof}

Next we record two lemmas which are consequences of Theorem \ref{maxint} and Proposition \ref{ifpp}(i). A novel contribution here is the elimination of rank one assumption, see \cite[Lemmas 3.10 and 3.11]{MB} or \cite[Corollary 15]{Tao}, as well as the introduction of signs inside the mixed characteristic polynomial.

\begin{lemma}\label{ind}
Let $X$ be a random positive semidefinite $d\times d$ matrix. Let $A_2,\ldots, A_m$ be $d\times d$ deterministic positive semidefinite matrices. Let $\ve_i=\pm 1$, $i=1,\ldots,m$. Then, with positive probability we have
\begin{equation}\label{ind1}
\maxroot(\mu[\ve_1 X,\ve_2 A_2,\ldots,\ve_m A_m]) \le \maxroot(\mu[\ve_1 \E{}{X},\ve_2 A_2,\ldots,\ve_m A_m]).
\end{equation}
Likewise, with positive probability we have
\begin{equation}\label{ind2}
 \maxroot(\mu[\ve_1 \E{}{X},\ve_2 A_2,\ldots,\ve_m A_m]) \le 
 \maxroot(\mu[\ve_1 X,\ve_2 A_2,\ldots,\ve_m A_m]).
\end{equation}
\end{lemma}

\begin{lemma}\label{max}
Suppose that $X_1, \dots, X_m$ are jointly independent random positive semidefinite $d\times d$ matrices which take finitely many values.  Let $\ve_i=\pm 1$, $i=1,\ldots,m$. Then with positive probability
\begin{equation}\label{max1}
\maxroot(\mu[\ve_1 X_1,\ldots \ve_m X_m]) \le  \maxroot(\mu[\ve_1 \E{}{X_1}, \ldots, \ve_m \E{}{X_m}]).
\end{equation}
\end{lemma}

The following lemma shows the monotonicity of the maximal root of a mixed characteristic polynomial. In contrast with Lemmas \ref{ind} and \ref{max}, this property has not been observed before even in the case when all signs $\ve_i=1$.

\begin{lemma}\label{mono}
Let $A_1,\ldots,A_m$ and $B_1$ be positive semidefinite hermitian $d\times d$ matrices such that $A_1 \le B_1$. Let $\ve_i=\pm 1$, $i=2,\ldots,m$.
Then,
\begin{equation}\label{norm2}
\maxroot({\mu[ A_1,\ve_2 A_2,\ldots,\ve_m A_m]}) \le  \maxroot(\mu[ B_1,\ve_2 A_2,\ldots,\ve_m A_m]).
\end{equation}
Moreover,
\begin{equation}\label{norm2b}
\maxroot({\mu[ -A_1,\ve_2 A_2,\ldots,\ve_m A_m]}) \ge  \maxroot(\mu[ -B_1,\ve_2 A_2,\ldots,\ve_m A_m]).
\end{equation}
\end{lemma}

\begin{proof}
For any $d\times d$ matrices $B_1,\ldots,B_m$ and $t\ne 0$, a simple calculation shows that
\[
\mu[t B_1, \ldots, t B_m](z) = t^d \mu[B_1,\ldots,B_m](z/t) \qquad\text{for any }z\in \C.
\]
Hence, for any $t>0$ we have
\begin{equation}\label{norm5}
\maxroot (\mu[t B_1, \ldots, t B_m]) = 
t \maxroot ( \mu[B_1, \ldots, B_m]) .
\end{equation}
Moreover, for any matrix $B$, by Jacobi's formula we have for $z\ne 0$,
\[
\mu[B](z) =z^d (1-\partial_{z_1}) \det (\mathbf I + z_1 z^{-1} B)|_{z_1=0} = z^d (1- \tr(z^{-1} B)) = z^{d-1}(z- \tr B).
\]
Consequently, for any hermitian matrix $B$ we have
\begin{equation}\label{norm6}
\maxroot \mu[B]= \tr B.
\end{equation}

Without loss of generality we can assume that $A_1 \ne B_1$ in addition to $A_1 \le B_1$. For any $0<p<1$, consider a random matrix $X_p$ which takes value $\frac{1}{p} A_1$ with probability $p$ and value $\frac{1}{p-1}(B_1-A_1)$ with probability $1-p$. Since
$\E{}{X_p} = B_1$, Lemma \ref{ind} yields that with positive probability
\begin{equation}\label{norm7}
\maxroot \mu[ X_p,\ve_2 A_2,\ldots,\ve_m A_m]  \le \maxroot \mu[ B_1,\ve_2 A_2,\ldots,\ve_m A_m] .
\end{equation}
By \eqref{norm5} and \eqref{norm6},
\[
\begin{aligned}
(1-p) & \maxroot  \mu\bigg[ \frac1{1-p}(B_1-A_1), \ve_2 A_2, \ldots, \ve_m A_m\bigg]  \\
 = & \maxroot \mu [B_1-A_1,(1-p)\ve_2 A_2, \ldots, (1-p) \ve_m A_m] \\
 \to & \maxroot \mu[B_1-A_1,\mathbf 0, \ldots,\mathbf 0] = \maxroot \mu[B_1-A_1] = \tr(B_1-A_1)>0 \qquad\text{as }p \to 1.
\end{aligned}
\]
Hence,  
\[
\maxroot  \mu\bigg[ \frac1{1-p}(B_1-A_1), \ve_2 A_2, \ldots, \ve_m A_m\bigg]  \to \infty \qquad \text{as }p \to 1.
\]
Combining this with \eqref{norm7} shows that for $p$ close to 1,
\[
\maxroot \mu\bigg [\frac{1}{p} A_1,\ve_2 A_2,\ldots,\ve_m A_m \bigg] \le \maxroot \mu[B_1,\ve_2 A_2,\ldots, \ve_m A_m].
\]
Letting $p\to 1$ yields \eqref{norm2}.

In order to prove \eqref{norm2b}, observe that 
Lemma \ref{ind} applied for the random variable $-X_p$ yields that with positive probability
\begin{equation}\label{norm7b}
\maxroot \mu[ -X_p,\ve_2 A_2,\ldots,\ve_m A_m]  \ge \maxroot \mu[ -B_1,\ve_2 A_2,\ldots,\ve_m A_m] .
\end{equation}
By a similar argument as above we have
\[
\maxroot  \mu\bigg[ \frac1{1-p}(A_1-B_1), \ve_2 A_2, \ldots, \ve_m A_m\bigg]  \to -\infty \qquad \text{as }p \to 1.
\]
Combining this with \eqref{norm7b} shows that for $p$ close to 1,
\[
\maxroot \mu\bigg [\frac{-1}{p} A_1,\ve_2 A_2,\ldots,\ve_m A_m \bigg] \ge \maxroot \mu[-B_1,\ve_2 A_2,\ldots, \ve_m A_m].
\]
Letting $p\to 1$ yields \eqref{norm2b}.
\end{proof}

\subsection{Multivariate barrier function}

We need a definition of a barrier function which plays a key role in the Marcus-Spielman-Srivastava proof \cite{MSS}.

\begin{definition}
Let $Q \in \R[z_1,\ldots,z_m]$ be a real stable polynomial. 
We say that a point $z\in\R^m$ is {\em above the roots} of $Q$ if 
\[
Q(z+t)>0\qquad\textrm{for all }t=(t_1,\ldots,t_m)\in\R^m, t_i\ge 0.
\]
For $i=1,\ldots,m$, we define the {\em barrier function of $Q$} in direction of variable $z_i$ at $z$ as
\[
\Phi^i_Q(z) = \partial_{z_i} \log Q(z)= \frac{ \partial_{z_i} Q(z)}{Q(z)}.
\]
\end{definition}

For our considerations, we need to recall two fundamental results from \cite{MSS}. The first result describes the analytic properties of the barrier function \cite[Lemma 5.8]{MSS}, see also \cite[Lemma 3.14]{MB}.

\begin{lemma}\label{bar}
Let $Q \in \R[z_1,\ldots,z_m]$ be a real stable polynomials. Let $z\in\R^m$ be above the roots of $Q$. Then, for any $1\le j \le m$, the barrier function $t\mapsto \Phi_Q^i(x+t\mathbf e_j)$ is a non-negative, non-increasing, and convex function of $t\ge 0$.
\end{lemma}

Using the multivariate barrier argument, Marcus, Spielman, and Srivastava \cite[Theorem 5.1]{MSS} have shown the following upper bound on the roots of the mixed characteristic polynomial. Note the assumption in \cite{MSS} that $\sum_{i=1}^m A_i = \bI$ can be relaxed by the inequality $\sum_{i=1}^m A_i \le \bI$ due to Lemma \ref{mono}.

\begin{theorem}\label{mixed}
Let $\epsilon>0$. Suppose $A_1, \ldots, A_m$ are $d \times d$ positive semidefinite matrices satisfying
\begin{equation}\label{mixed1}
\sum_{i=1}^m A_i \le \bI 
\qquad\text{and}\qquad
\tr(A_i) \le \epsilon \quad\text{for all i}.
\end{equation}
Then, all roots of the mixed characteristic polynomial $\mu[A_1,\ldots, A_m]$ are real and the largest root
is at most $(1 + \sqrt{\epsilon})^2$.
\end{theorem}

The proof of Theorem \ref{mixed} relies on results controlling roots of a polynomial $(1-\partial_{z_i})Q$ in terms of roots of a real stable polynomial $Q$ and a barrier function $\Phi^i_Q$, see \cite[Lemma 5.10 and Lemma 5.11]{MSS}. Here, we shall show a simple converse of these lemmas.

\begin{lemma}\label{con}
Let $Q \in \R[z_1,\ldots,z_m]$ be a real stable polynomial,  $c\in \R$, and $i=1,\ldots,m$. If $z\in \R^m$ is above roots of $(1+c\partial_{z_i})Q$, then $z+c\mathbf e_i$ is above roots of $Q$.
\end{lemma}

\begin{proof} First, we reduce the proof to univariate case. Fix any $t_j \ge 0$, $j\in [m] \setminus \{i\}$. Consider univariate restriction of $Q$ given by
\[
q(t) = Q(z_1+t_1, \ldots, z_{i-1}+t_{i-1}, z_i+t, z_{i+1}+t_{i+1}, \ldots, z_m+t_m), \qquad t\in \R.
\]
Since $q$ is a restriction of a real stable polynomial $Q$, a polynomial $q$ is real-rooted.
It suffices to show that $c$ is above roots of $q$. Indeed, since $t_j \ge 0$ are arbitrary, this implies that $z+c\mathbf e_i$ is above roots of $Q$.

We can assume that $c\ne 0$, since the conclusion is trivial for $c=0$.
Since $z$ is above the roots of $Q$, the leading coefficient of $q$ is positive. By our hypothesis, $0$ is above roots of $q+cq'$. Hence, 
\begin{equation}\label{qqp}
q(t)> -c q'(t) \qquad\text{for all } t \ge 0.
\end{equation}
Let $\lambda_1, \ldots, \lambda_r$ be all roots of $q$ with multiplicities. Since $q$ is a restriction of a real stable polynomial $Q$, they are all real. The proof now splits depending on the sign of $c$.

First suppose that $c<0$. If $q(t_0)=0$ for some $t_0 \ge 0$, then \eqref{qqp} would imply that $q(t) \le 0$ for all $t \ge t_0$. This contradicts the fact that the leading coefficient of $q$ is positive. Hence, $0$ is above the roots of $q$. In other words, $\maxroot q = \max_j \lambda_j<0$ and $q(t) > 0 $ for all $t\ge 0$.
It remains to show that actually $c$ is above roots of $q$.
By \eqref{qqp} we have
\[
\frac{-c q'(t)}{q(t)} = \sum_{j=1}^r \frac{-c}{t-\lambda_j}<1 \qquad\text{for }t\ge 0.
\]
Setting $t=0$ implies that $\lambda_j<c$ for all $j=1,\ldots,r$. Hence, $c$ is above roots of $q$.

Next suppose that $c>0$. Without loss of generality, we can assume that $\lambda_1= \maxroot q = \max_i \lambda_i$. If $\lambda_1 \le 0$, then there is nothing to show. Otherwise, \eqref{qqp} implies that $q(t)<0$ for all $0 \le t <\lambda_1$. Hence, $\lambda_j<0$ for all $j=2,\ldots,m$. By \eqref{qqp} we have
\[
1<\frac{-c q'(t)}{q(t)} = \sum_{j=1}^r \frac{-c}{t-\lambda_j}< \frac{-c}{t-\lambda_1} \qquad\text{for }0 \le t <\lambda_1.
\]
Setting $t=0$ implies that $\lambda_1<c$. Hence, $c$ is above roots of $q$.
\end{proof}

Next we shall establish above the roots properties of multivariate polynomials appearing in Definition \ref{dmcp} of mixed characteristic polynomials. For a collection of $d\times d$ matrices $A_1,\ldots, A_m$, define its determinantal polynomial
\begin{equation}\label{Q0}
Q_0(x,z_1,\ldots,z_m)= \det\bigg( x \mathbf I + \sum_{i=1}^m z_i A_i \bigg) \in \C[x,z_1,\ldots,z_m].
\end{equation}
 For a fixed choice of signs $\ve_i = \pm 1$, define polynomials
\begin{equation}\label{Qm}
Q_k= \prod_{i=1}^k (1-\ve_i\partial_{z_i}) Q_0, \qquad k=1, \ldots,m.
\end{equation}

\begin{lemma}\label{abp} Let $A_1,\ldots,A_m$ be 
positive semidefinite hermitian $d\times d$ matrices. Let $\ve_i=\pm 1$ for $i=1,\ldots,m$.
If $x_0>0$ and $x_0> \maxroot \mu[\ve_1A_1,\ldots,\ve_mA_m]$, then the point $(x_0,0,\ldots,0)\in \R^{m+1}$ is above roots of the polynomial $Q_m$ given by \eqref{Qm}.
\end{lemma}

\begin{proof} By Definition \ref{dmcp} and \eqref{Qm} for any $k=1,\ldots,m$ we have
\begin{equation}\label{Qm2}
\mu[\ve_1A_1,\ldots,\ve_k A_k](z) = Q_k(z,0,\ldots,0), \qquad z\in\C.
\end{equation}
By permuting matrices $A_i$ we can assume that for some $l=0,\ldots, m+1$ we have 
\[
\ve_i=\begin{cases} -1 & i=1,\ldots,l,
\\
1 & i=l+1,\ldots,m.
\end{cases}
\]
Here, $l=0$ or $l=m$ means that all $\ve_i$ are equal to $1$ or $-1$, respectively.
By Lemma \ref{mono} for any $k=l,\ldots,m$, we have 
\[
\begin{aligned}
\maxroot \mu[\ve_1 A_1,\ldots, \ve_m A_m] &\ge \maxroot \mu[\ve_1 A_1,\ldots,\ve_k A_k,\mathbf 0,\ldots,\mathbf 0] 
\\
&= \maxroot \mu[\ve_1 A_1,\ldots,\ve_k A_k],
\end{aligned}
\]
where $\mathbf 0$ is $d\times d$ zero matrix. Likewise, for any $k=1,\ldots,l$,
\[
\maxroot \mu[\ve_1 A_1,\ldots, \ve_k A_k] \le \maxroot \mu[\mathbf 0,\ldots,\mathbf 0] =0.
\]
Hence, the assumption that $x_0> \maxroot \mu[\ve_1A_1,\ldots,\ve_mA_m]$ and $x_0>0$, and \eqref{Qm2} yield that
\begin{equation}\label{qk}
Q_{k}(x_0,0,\ldots,0) >0 \qquad \text{for } k=1,\ldots,m.
\end{equation}

By induction we will show that $(x_0,0 \ldots,0)$ is above roots of $Q_k$ for any $k=0,1,\ldots,m$. Since $x_0>0$, the point $(x_0,0,\ldots,0)$ is above roots of $Q_0$. Suppose that the induction hypothesis is true for $Q_k$ some $k=0,1,\ldots,m-1$. That is, for any $t_j \ge 0$, $j=1,\ldots,m$, we have
\begin{equation}\label{qk3}
Q_{k}(x_0,t_1,\ldots,t_m) >0.
\end{equation}
We claim that
\begin{equation}\label{qk2}
\frac{(1-\ve_{k+1} \partial_{z_{k+1}})Q_k(x_0,t_1,\ldots, t_m)}{Q_k(x_0,t_1,\ldots, t_m)} = 1- \ve_{k+1}\Phi^{k+1}_{Q_k}(x_0,t_1,\ldots, t_m) >0.
\end{equation}
Indeed, by \eqref{qk} and $Q_{k+1}=(1-\ve_{k+1}\partial_{z_{k+1}})Q_k$, the above is true when $(t_1,\ldots,t_m)= (0, \ldots, 0)$. When $\ve_{k+1}=1$, by the monotonicity of the barrier function in Lemma \ref{bar}, we deduce the inequality \eqref{qk2} holds for arbitrary $t_j \ge 0$, $j=1,\ldots,m$. When $\ve_{k+1}=-1$, we deduce the same conclusion using the non-negativity of the barrier function. Either way, we conclude that $(x_0,0,\ldots,0)$ is above roots of $Q_{k+1}$.
\end{proof}

The following result is the last ingredient needed for the proof of Theorem \ref{thmp}. The proof of Lemma \ref{norm} is somewhat reminiscent of the proof of \cite[Theorem 5.1]{MSS}, which inductively deduces the above roots property of $Q_{k+1}$ from the same property for $Q_k$ and the convexity of barrier functions.

\begin{lemma}\label{norm}
If $A_1,\ldots,A_m$ are positive semidefinite hermitian $d\times d$ matrices and $\ve_i=\pm 1$, then
\begin{equation}\label{norm3}
\bigg\| \sum_{i=1}^m \ve_i A_i \bigg\| \le \maxroot( \mu[\ve_1 A_1,\ldots, \ve_m A_m] \cdot  \mu[-\ve_1 A_1,\ldots, -\ve_m A_m]).
\end{equation}
\end{lemma}

\begin{proof} 
Take any $x_0>  \maxroot( \mu[\ve_1 A_1,\ldots, \ve_m A_m] \cdot  \mu[-\ve_1 A_1,\ldots, -\ve_m A_m])$. 
Observe that $\mu[-\ve_1 A_1,\ldots, -\ve_m A_m] (z) = (-1)^d \mu[\ve_1 A_1,\ldots, \ve_m A_m] (-z)$
for any $z\in \C$. Hence,
\[
\begin{aligned}
\maxroot \mu[-\ve_1 A_1,\ldots, -\ve_m A_m] & = - \minroot \mu[\ve_1 A_1,\ldots, \ve_m A_m] 
\\
&\ge - \maxroot \mu[\ve_1 A_1,\ldots, \ve_m A_m] .
\end{aligned}
\]
This implies that $x_0>0$.

Let $Q_0$ be the polynomial given by \eqref{Q0}.
 By Lemma \ref{abp} the point $(x_0,0,\ldots,0) \in \R^{m+1}$ is above roots of polynomial $Q_m$ given by \eqref{Qm}. We shall apply repeatedly Lemma \ref{con}. Since $Q_{m}=(1-\ve_m \partial_{z_m})Q_{m-1}$, the point $(x_0,0, \ldots, 0,-\ve_m)$ is above roots of $Q_{m-1}$.  Applying Lemma \ref{con} for $Q_{m-1}=(1-\ve_{m-2} \partial_{z_{m-1}})Q_{m-2}$ shows the point $(x_0,0, \ldots, 0,-\ve_{m-1},-\ve_m)$ is above roots of $Q_{m-2}$. Repeating this argument shows that the point $(x_0,-\ve_1,\ldots,-\ve_m) \in \R^{m+1}$ is above roots of $Q_0$. Hence,
\begin{equation}\label{norm15}
Q_0(x,-\ve_1,\ldots,-\ve_m) = \det \bigg(x \mathbf I - \sum_{i=1}^m \ve_i A_i \bigg) >0 \qquad\text{for all }x\ge x_0.
\end{equation}
Using the assumption that $x_0> \maxroot \mu[-\ve_1 A_1,\ldots, -\ve_m A_m]$, by the symmetric argument we deduce that
\begin{equation}\label{norm17}
\det \bigg(x \mathbf I + \sum_{i=1}^m \ve_i A_i \bigg) >0 \qquad\text{for all }x\ge x_0.
\end{equation}
Since matrices $A_i$ are hermitian, \eqref{norm15} and \eqref{norm17} yield  
\[
x_0 > \maxroot\det \bigg(x \mathbf I - \sum_{i=1}^m \ve_i A_i \bigg)\bigg(x \mathbf I + \sum_{i=1}^m \ve_i A_i \bigg) =
\bigg\| \sum_{i=1}^m \ve_i A_i \bigg\|.
\]
This proves \eqref{norm3}.
\end{proof}

As an immediate corollary of Lemma \ref{norm}, we have the following norm bound for the sum of positive semidefinite matrices $A_1,\ldots, A_m$ 
\begin{equation}\label{norm10}
\bigg\| \sum_{i=1}^m  A_i \bigg\| \le \maxroot \mu[ A_1,\ldots, A_m] .
\end{equation}
We are now ready to prove Theorem \ref{thmp}. Note the proof does not use the full strength of our results as it relies on the special case of interlacing family of polynomials when all signs $\ve_1=\ldots=\ve_m=1$ in Lemma \ref{max}.

\begin{proof}[Proof of Theorem \ref{thmp}]
The assumption \eqref{thmp1} translates into the assumption \eqref{mixed1} for $A_i=\E{}{X_i}$. Applying \eqref{norm10}, Lemma \ref{max}, and then Theorem \ref{mixed} yields with positive probability
\[
 \left\| \sum_{i=1}^m X_i \right\| \le \maxroot(\mu[X_1,\ldots, X_m]) \le \maxroot(\mu[A_1,\ldots, A_m]) \le (1+\sqrt{\epsilon})^2.
\]
Actually, the first inequality holds with probability $1$, the last inequality is deterministic, and only the middle inequality happens with positive probability by the virtue of Lemma \ref{max}.
\end{proof}

It is worth adding that Theorem \ref{thmp} can be improved when the control on ranks of random matrices $X_i$ is assumed. Theorem \ref{mixed2} was first shown by the author, Casazza, Marcus, and Speegle \cite[Theorem 1.5]{BCMS} when rank $k=2$. The bound for higher ranks was shown by Xu, Xu, and Zhu \cite[Theorem 1.9]{XXZ2} using multivariate barrier argument introduced by Leake and Ravichandran \cite{RL}. A similar result in a general setting of hyperbolic polynomials and hyperbolicity cones, albeit with weaker bounds for small values of $\epsilon$, was shown by Br\"anden \cite[Theorem 5.6]{Bra}.

\begin{theorem}\label{mixed2}
Let $k\in \N$ and $0<\epsilon\le (k-1)^2/k$. Suppose $A_1, \ldots, A_m$ are $d \times d$ positive semidefinite matrices of rank at most $k$ satisfying
\begin{equation*}
\sum_{i=1}^m A_i \le \bI 
\qquad\text{and}\qquad
\tr(A_i) \le \epsilon \quad\text{for all i}.
\end{equation*}
Then, 
\[
\maxroot \mu[A_1,\ldots, A_m]\le \bigg( \sqrt{1-\frac{\epsilon}{k-1}}+\sqrt{\epsilon}\bigg)^2.
\]
\end{theorem}

Using the same argument as in the proof of Theorem \ref{thmp}, we deduce the following corollary of Theorem \ref{mixed2}.

\begin{corollary} Let $k\in \N$ and $0<\epsilon\le (k-1)^2/k$. If the random matrices $X_i$ have ranks at most $k$ and satisfy all the assumptions of Theorem \ref{thmp} then 
\begin{equation}
\P{}{\bigg\| \sum_{i=1}^m X_i \bigg\| \leq \bigg( \sqrt{1-\frac{\epsilon}{k-1}}+\sqrt{\epsilon}\bigg)^2
 } > 0.
\end{equation}
\end{corollary}

\section{Higher rank improvement of Kyng-Luh-Song discrepancy result}\label{S3}

In this section we give the proof of Theorem \ref{kls}. This is a generalization of a matrix discrepancy result due to Kyng, Luh, and Song \cite{KLS}, which strengthens Weaver's KS$_2$ conjecture shown by Marcus, Spielman, and Srivastava \cite{MSS}, and its generalization by Akemann and Weaver \cite{AW}. We show that the rank one assumption in the matrix discrepancy result \cite[Theorem 1.4]{KLS} can be dropped.

In a nutshell, the proof of Theorem \ref{kls} involves the product $f_{\ve_1,\ldots,\ve_m}$ of two mixed characteristic polynomials as in \eqref{qint0}. By Lemma \ref{norm}, the maximum root of $f_{\ve_1,\ldots,\ve_m}$ controls the norm $\| \sum_{i=1}^m \ve_i A_i \|$ of the linear combination of positive definite matrices $A_1,\ldots,A_m$. Since $\{f_{\ve_1,\ldots,\ve_m}\}$ forms an interlacing family of polynomials, it suffices to estimate the maximum root of the expected polynomial $\E{}{f_{\xi_1,\ldots,\xi_m}}$. In the case when matrices $A_1,\ldots,A_m$ have rank one, this reduces to the expected characteristic polynomial from \cite[Proposition 5.4]{KLS} given by
\[
\E{}{\det\bigg( x^2 \mathbf I - \big( \sum_{i=1}^m (\xi_i - \E{}{\xi_i}) A_i \big)^2\bigg)}.
\]
For higher rank matrices, we introduce a quadratic analogue of mixed characteristic polynomial.

\begin{definition} \label{dmcp2}
Let $A_1, \ldots, A_m$ be $d\times d$ matrices. A quadratic mixed characteristic polynomial is defined for $x\in \C$ by
\begin{equation*}
\mu_2[A_1,\ldots,A_m](x) =  
\bigg(\prod_{i=1}^m (1 - \partial_{z_i}\partial_{w_i}) \bigg)
 \det \bigg( x \mathbf I + \sum_{i=1}^m z_i A_i\bigg)
 \det \bigg( x \mathbf I + \sum_{i=1}^m w_i A_i \bigg)
  \bigg|_{\genfrac{}{}{0pt}{2}{z_1=\ldots=z_m=0}{w_1=\ldots=w_m=0}}.
\end{equation*}
\end{definition}

The following lemma is an analogue of Lemma \ref{mux}.

\begin{lemma}\label{muxx}
Let $A_1, \ldots, A_m$ be $d\times d$ positive semidefinite matrices. Let $\xi_1,\ldots ,\xi_m$ be jointly independent random variables, which take finitely many real values.
Define a random polynomial
\begin{equation}\label{muxx0}
f_{\xi_1,\ldots,\xi_m}(x) = \mu[\xi_1 A_1,\ldots,\xi_m A_m](x)\mu[-\xi_1 A_1,\ldots,-\xi_m A_m](x) \in \R[x].
\end{equation}
 Then, $\E{}{f_{\xi_1,\ldots, \xi_m}}$ is a real-rooted and monic polynomial of degree $2d$. Moreover, if $\E{}{\xi_i}=0$ for all $i=1,\ldots,m$, then
\begin{equation}\label{muxx1}
\E{}{f_{\xi_1,\ldots, \xi_m}}
= \mu_2[\var{\xi_1}^{1/2} A_1,\ldots, \var{\xi_m}^{1/2} A_m].
\end{equation}
\end{lemma}

\begin{proof} The fact that $\E{}{f_{\xi_1,\ldots, \xi_m}}$ is a real-rooted polynomial follows from Theorem \ref{qint}. The formula \eqref{muxx1} is then a consequence of \eqref{qint1} and \eqref{qint2}
\[
\begin{aligned}
 \E{}{f_{\xi_1,\ldots, \xi_m}} = &
\bigg(\prod_{i=1}^m (1 - \E{}{\xi_i^2}\partial_{z_i}\partial_{w_i})
\bigg)
\det \bigg( x \mathbf I + \sum_{i=1}^m z_i A_i\bigg)
 \det \bigg( x \mathbf I + \sum_{i=1}^m w_i A_i \bigg)
  \bigg|_{\genfrac{}{}{0pt}{2}{z_1=\ldots=z_m=0}{w_1=\ldots=w_m=0}}
\\
= &
\bigg(\prod_{i=1}^m (1 - \partial_{z_i}\partial_{w_i}) \bigg)
\\
& \det \bigg( x \mathbf I + \sum_{i=1}^m z_i \var{\xi_i}^{1/2} A_i\bigg)
 \det \bigg( x \mathbf I + \sum_{i=1}^m w_i \var{\xi_i}^{1/2} A_i \bigg)
  \bigg|_{\genfrac{}{}{0pt}{2}{z_1=\ldots=z_m=0}{w_1=\ldots=w_m=0}}.
\end{aligned} \qedhere
  \]
\end{proof}

Our goal is to establish the bound on the largest root of a quadratic mixed polynomial in Definition \ref{dmcp2}.
Anari and Oveis Gharan \cite[Lemma 4.8]{AG} have shown a crucial result which provides the control of the barrier function of $(1-\partial^2_{z_j})p$ in terms of the barrier function of a real stable polynomial $p$. It is an analogue of a result due to Marcus, Spielman, and Srivastava \cite[Lemma 5.11]{MSS}, which controls the barrier function of $(1-\partial_{z_j})p$ in terms of the barrier function of $p$. The following is a slight variation of this result for the operator $p \mapsto (1-c^2\partial^2_{z_j})p$, which was shown by Kyng, Luh, and Song \cite[Lemma A.1]{KLS}. The assumption that $c\in [0,1]$ in \cite{KLS} can be relaxed by $c\ge 0$ as shown by a rescaling argument below.

\begin{lemma}\label{AG} Let $p\in \R[z_1,\ldots,z_m]$ be stable and let $c\ge 0$. Suppose that $z\in \R^m$ lies above the roots of $p$ and
\begin{equation}\label{AG1}
c^2\bigg(\frac 2{\delta_j} \Phi^j_p(z) + (\Phi^j_p(z))^2 \bigg) \le 1
\qquad\text{for some }j\in [m]\text{ and }\delta_j>0.
\end{equation}
Then, $z+\delta_j \mathbf e_j$ lies above the roots of $(1-c^2\partial^2_{z_j})p$ and
\begin{equation}\label{AG2}
\Phi^i_{(1-c^2\partial^2_{z_j})p}(z+\delta_j \mathbf e_j) \le \Phi^i_p(z)
\qquad\text{for all }i\in [m].
\end{equation}
\end{lemma}

\begin{proof} The case $c=0$ is trivial, so we can assume that $c>0$. The case $c=1$ was shown in \cite[Lemma 4.8]{AG}. The general case follows by a rescaling argument. Define polynomial $p_c(z_1,\ldots,z_m)=p(cz_1,\ldots,c z_m)$. Observe that $w\in \R^m$ is above roots of $p$ if and only if $c^{-1}w$ is above roots of $p_c$. A simple calculation shows that 
\[
(1-\partial^2_{z_j})p_c(w)=(1-c^2\partial^2_{z_j})p(cw) \qquad w\in \R^m.
\]
Moreover, if $w$ is above the roots of $p_c$ or $z(1-c^2\partial^2_{z_j})p_c$, respectively, then
\[
\Phi^j_{p_c}(w)= c \Phi^j_p(cw) \quad\text{and}\quad
\Phi^j_{(1-\partial^2_{z_j})p_c}(w)= c \Phi^j_{(1-c^2\partial^2_{z_j})p}(cw). 
\]
Thus, $c^{-1}z$ is above the roots of $p_c$ and by \eqref{AG1} we have
\[
\frac {2c}{\delta_j} \Phi^j_{p_c}(c^{-1}z) + (\Phi^j_{p_c}(c^{-1}z))^2 \le 1.
\]
By \cite[Lemma 4.8]{AG}, $z+c^{-1}\delta_j\mathbf e_j$ lies above the roots of $(1-\partial^2_{z_j})p_c$ and
\[
\Phi^i_{(1-\partial^2_{z_j})p_c}(c^{-1}z+c^{-1}\delta_j \mathbf e_j) \le \Phi^i_{p_c}(c^{-1}z)
\qquad\text{for all }i\in [m].
\]
This yields \eqref{AG2}.
\end{proof}

Iterating Lemma \ref{AG} with $c=1/\sqrt{2}$ yields the following corollary.

\begin{corollary} \label{AG10}
Let $p\in \R[z_1,\ldots,z_m]$ be stable. Let $S \subset [m]$. Suppose that $z\in \R^m$ lies above the roots of $p$ and for some $\delta_j>0$, $j\in S$, we have
\[
\frac 1{\delta_j} \Phi^j_p(z) + \frac 12 (\Phi^j_p(z))^2 \le 1
\qquad\text{for all }j \in S.
\]
Then, the point $z+\sum_{j\in S} \delta_j \mathbf e_j$ lies above the roots of $\prod_{j\in S} (1-\tfrac 12\partial^2_{z_j})p$.
\end{corollary}

\begin{proof}
By permuting variables $z_1,\ldots, z_m$ if necessary, we can assume that $S=\{1,\ldots,n\}$ for some $n\in [m]$. For $k=0,\ldots,n$, define a polynomial $p_k$ and point $y_k$ by
\[
p_k=\prod_{j=1}^k (1-\tfrac 12\partial^2_{z_j})p,
\qquad
y_k= z+ \sum_{j=1}^k \delta_j \mathbf e_j \in \R^m.
\]
By convention, $p_0=p$ and $y_0=z$. Then, using Lemma \ref{AG} and the monotonicity of the barrier function we show inductively that $y_k$ lies above the roots of $p_k$ for all $k \in [m]$.
\end{proof}

Using Corollary \ref{AG10} we obtain the following result, which extends the barrier argument in the proof of \cite[Proposition 5.4]{KLS} to matrices of higher rank.

\begin{lemma}\label{qx}
 Let $B_1, \ldots, B_m$ be $d \times d$ positive semidefinite matrices such that
\begin{equation}\label{qx0}
\max_{i} \tr(B_i) \le 1 
\quad\text{and}\quad
\sum_{i=1}^m  \tr(B_i) B_i \le \mathbf I.
\end{equation}
Let
\begin{equation}\label{qx1}
Q(x,z_1,\ldots,z_m,w_1,\ldots,w_m)= \det \bigg( x \mathbf I + \sum_{i=1}^m z_i B_i \bigg)\det \bigg( x \mathbf I + \sum_{i=1}^m w_i B_i \bigg).
\end{equation}
Then, for any subset $S \subset [m]$, the point $(4,0,\ldots,0)$ lies above the roots of the polynomial
\begin{equation}\label{qx2}
\prod_{i\in S} (1 - \tfrac 12\partial^2_{z_i\equiv w_i})Q.
\end{equation}
Here, $\partial^2_{z_i\equiv w_i}$ denotes the composition of diagonalization operator that identifies variables $z_i$ and $w_i$, which is followed by the second order partial derivative $\partial^2_{z_i}$ in variable $z_i$.
\end{lemma}

\begin{proof}
By Lemma \ref{spd} the operator $(1-\tfrac 12\partial^2_{z_i})=(1-\tfrac 1{\sqrt{2}}\partial_{z_i})(1+\tfrac 1{\sqrt{2}}\partial_{z_i})$ preserves stability. Since diagonalization operator, which identifies variables $z_i$ and $w_i$, also preserves stability, so does the operator $(1 - \tfrac 12\partial^2_{z_i\equiv w_i})$. Moreover, the operators $(1 - \tfrac 12\partial^2_{z_i\equiv w_i})$, $i=1,\ldots,m$, commute. Since $Q$ is a real stable polynomial, so is the polynomial \eqref{qx2}.

Let $0<t<\alpha$ be constants to be determined later. As we will see $\alpha=4$ and $t=2$ turn out to be optimal choices. Define $\delta_i= t \tr(B_i)$, $i=1,\ldots,m$. We claim that the point with coordinates $x=\alpha$, $z_i=w_i=-\delta_i$ for $i\in [m]$, is above the roots of the polynomial \eqref{qx1}. Indeed, take $x\ge \alpha$ and $z_i, w_i \ge -\delta_i$ for $i\in [m]$. Then, by \eqref{qx0} we have
\begin{equation}\label{qx4}
\alpha \mathbf I + \sum_{i=1}^m z_i B_i \ge 
\alpha \mathbf I - \sum_{i=1}^m \delta_i B_i \ge
\alpha \mathbf I - t \sum_{i=1}^m \tr(B_i) B_i \ge (\alpha-t) \mathbf I.
\end{equation}
An analogous calculation in variables $w_i$, $i\in [m]$, shows that 
\[
Q(x,z_1,\ldots,z_m,w_1,\ldots,w_m) \ge \det((\alpha-t)\mathbf I)^2 = (\alpha-t)^{2d} >0.
\]

Let $p$ be a polynomial obtained from the polynomial $Q$ by identifying variables $z_i$ and $w_i$ for $j\in S$. By rearranging variables if necessary, we can assume that $S=[n]$ for some $n\in [m]$. Then,
\[
p(x,z_1,\ldots,z_m, w_{n+1}, \ldots,w_m) = 
\det \bigg( x \mathbf I + \sum_{i=1}^m z_i B_i \bigg)\det \bigg( x \mathbf I + \sum_{i=1}^m z_i B_i + \sum_{i=n+1}^m w_i B_i \bigg).
\]
Let $j\in S$. By the product rule and Jacobi's formula we have
\[
\begin{aligned}
\Phi^j_p= 
\frac{\partial_{z_j}p}{p} = &
\frac{1}{p}
\partial_{z_i} \det \bigg( x \mathbf I + \sum_{i=1}^m z_i B_i \bigg)\det \bigg( x \mathbf I + \sum_{i=1}^m z_i B_i + \sum_{i=n+1}^m w_i B_i \bigg)
\\
&+ \frac{1}{p}
 \det \bigg( x \mathbf I + \sum_{i=1}^m z_i B_i \bigg) \partial_{z_i}\det \bigg( x \mathbf I + \sum_{i=1}^m z_i B_i + \sum_{i=n+1}^m w_i B_i \bigg)
 \\
= & \tr\bigg(\bigg( x \mathbf I + \sum_{i=1}^m z_i B_i \bigg)^{-1} B_j\bigg) + 
\tr\bigg(\bigg( x \mathbf I + \sum_{i=1}^m z_i B_i + \sum_{i=n+1}^m w_i B_i \bigg)^{-1} B_j\bigg).
\end{aligned}
\]
By the operator monotonicity of the inverse map $X \mapsto X^{-1}$ on positive definite matrices and \eqref{qx4}, we have
\[
\tr\bigg(\bigg( x \mathbf I + \sum_{i=1}^m z_i B_i \bigg)^{-1} B_j\bigg) \le \frac{ \tr( B_j)}{\alpha -t}  = \frac{\delta_j}{(\alpha-t) t}.
\]
A similar estimate holds for the second term.
Thus, evaluating the barrier function at the point $\mathbf z$ with coordinates $x=\alpha$, $z_i=-\delta_i$, $i\in [m]$, and $w_i=-\delta_i$, $i \in [m] \setminus [n]$, yields
\[
\Phi^j_p(\mathbf z) \le \frac{ 2\tr( B_j)}{\alpha -t} = \frac{2\delta_j}{(\alpha-t) t}.
\]
Recall that the point $\mathbf z$ lies above the roots of $p$.
In order to apply Corollary \ref{AG10}, for any $j\in S$, we must require that 
\[
\frac 1{\delta_j} \Phi^j_p(\mathbf z) + \frac 12 (\Phi^j_p(\mathbf z))^2 \le 1.
\]
Since $ \tr( B_j) \le 1$, the above follows from
\[
\frac{2}{(\alpha-t) t}+ \frac{2}{(\alpha-t)^2} \le 1.
\]
An elementary calculation shows that the smallest value of $\alpha$ for which this holds is $\alpha=4$ and the corresponding $t=2$. By Corollary \ref{AG10}, the point $\mathbf z+ \sum_{j\in S} \delta_j \mathbf e_{z_j}$, with coordinates $x=4$, $z_i=0$ for $i\in [n]$, $z_i=w_i=-\delta_i$ for $i\in [m] \setminus [n]$, is above the roots of the polynomial $\prod_{j\in S} (1-\tfrac 12\partial^2_{z_j})p$, which coincides with the polynomial \eqref{qx2}. Consequently, $(4,0,\ldots,0)$ is also above the roots of the polynomial \eqref{qx2}. 
\end{proof}

To prove next lemma we will use the Helton-Vinnikov theorem \cite{HV} in a convenient formulation given by Borcea and Br\"and\'en \cite[Corollary 6.7]{BB}. The proof of the second part of Theorem \ref{HV} is elementary and can be found in \cite[Corollary 3.4]{BCMS} or \cite[Remark 5.6]{MSS}.

\begin{theorem}\label{HV}
Let $p \in \R[x,y]$ be a real polynomial of degree $d \in \N$. Then $p$ is stable if and only if there exists two  positive semidefinite $d \times d$ matrices $A,B$ and a hermitian $d \times d$ matrix $C$ such that 
\begin{equation}\label{HV1}
p(x,y)=\pm \det( xA +y B +C).
\end{equation}
In addition, if the point $(0,0)$ us above the roots of $p$, then the sign in \eqref{HV1} is positive and both $A+B$ and $C$ are positive definite.
\end{theorem}

\begin{lemma}\label{qz} Suppose $p \in \R[x_1,\ldots,x_n,z,w]$ be a real stable polynomial. Let $(t_1,\ldots,t_n)\in \R^n$. If $(t_1,\ldots,t_n,0,0)$ is above the roots of $p$ and $(t_1,\ldots,t_n)$ is above the roots of the polynomial
\begin{equation}\label{qz1}
(1 - \tfrac 12 \partial^2_{z \equiv w})p \bigg|_{z=0} \in \R[x_1,\ldots,x_n],
\end{equation}
then $(t_1,\ldots,t_n)$ is also above the roots of the polynomial
\begin{equation}\label{qz2}
(1 - \partial_{z} \partial_{w}) p \bigg|_{z=w=0} \in \R[x_1,\ldots,x_n].
\end{equation}
\end{lemma}

\begin{proof} Fix a point $(x_1,\ldots,x_n)\in \R^n$ such that $x_i \ge t_i$ for all $i\in [n]$. Consider a real stable polynomial $q(z,w)= p(x_1,\ldots,x_n,z,w)$. If $q$ is a constant polynomial, the the conclusion is trivial. Otherwise, by Theorem \ref{HV} we can represent $q(z,w)= \det (z A+w B+ C)$ for some positive semidefinite matrices $A$, $B$, and hermitian matrix $C$. By \cite[Remark 5.6]{MSS}, the matrix $A+B$ is positive definite. Define matrices $\tilde A= C^{-1/2}AC^{-1/2}$ and $\tilde B= C^{-1/2} B C^{-1/2}$. Then,
\[
\begin{aligned}
q(z,w)= \det (z A+w B+ C) &=\det (z C^{1/2}\tilde A C^{1/2} +w C^{1/2}\tilde B C^{1/2} + C)
\\
&= \det(C) \det(z \tilde A + w \tilde B + \mathbf I).
\end{aligned}
\]
By Jacobi's formula
\[
\partial_z q (z,w) = \det(C) \det(M) \tr(M^{-1}\tilde A)
\qquad\text{where }M=z\tilde A+ w \tilde B + \mathbf I.
\]
By the formula for the derivative of the inverse matrix \cite[Lemma 2.13]{AG} we have
\[
\partial_w M^{-1} = - M^{-1} B M^{-1}.
\]
Hence,
\[
\partial_w \partial_z q (z,w) = \det(C) ( \det(M) \tr(M^{-1}\tilde B) \tr(M^{-1}\tilde A) - \det(M) \tr(M^{-1}\tilde BM^{-1} \tilde A)).
\]
Thus,
\[
(1-\partial_w \partial_z) q \bigg|_{z=w=0}= \det(C) (1- \tr(\tilde A)\tr(\tilde B) + \tr(\tilde B \tilde A)).
\]
Likewise, letting $\tilde M=z(\tilde A+ \tilde B) + \mathbf I$ we have
\[
\begin{aligned}
\partial_z (q(z,z)) & = \det(C) \det({\tilde M}) \tr({\tilde M}^{-1}(\tilde A+\tilde B)),
\\
\partial^2_z (q(z,z)) & = \det(C) ( \det(\tilde M) (\tr({\tilde M}^{-1}(\tilde A+\tilde B)))^2  - \det(M) \tr({\tilde M}^{-1}(\tilde A+ \tilde B){\tilde M}^{-1} (\tilde A+\tilde B))
\end{aligned}
\]
Thus,
\[
\begin{aligned}
(1 &- \tfrac12 \partial^2_z)(q(z,z))\bigg|_{z=0} =\det(C)(1-\tfrac 12 (\tr(\tilde A + \tilde B))^2+ \tfrac 12 \tr((\tilde A + \tilde B)^2))
\\
&=\det(C)(1- \tr(\tilde A)\tr(\tilde B) + \tr(\tilde A \tilde B) + \tfrac 12( \tr(\tilde A^2) - (\tr\tilde A)^2 + \tr(\tilde B^2) - (\tr \tilde B)^2)).
\end{aligned}
\]
Note that for every positive semidefinite matrix we have $\tr(P^2) \le (\tr P)^2$, with equality happening if and only if rank of $P$ is $\le 1$. Therefore,
\[
(1 - \tfrac12 \partial^2_z)(q(z,z))\bigg|_{z=0} \le (1-\partial_w \partial_z) q \bigg|_{z=w=0}.
\]
Hence, if the point $(t_1, \ldots, t_n)$ is above the roof of the polynomial \eqref{qz1}, the the same point is also above the roots of the polynomial \eqref{qz2}.
\end{proof}

Finally, we can establish the bound on the largest root of the quadratic mixed characteristic polynomial.

\begin{theorem}\label{qxx}
 Let $B_1, \ldots, B_m$ be $d \times d$ positive semidefinite matrices satisfying \eqref{qx0}. Then, the largest root of the polynomial 
\[
\mu_2[B_1,\ldots,B_m](x)= \bigg(\prod_{i=1}^m (1- \partial_{w_i}\partial_{z_i})\bigg) \det \bigg( x \mathbf I + \sum_{i=1}^m z_i B_i \bigg)\det \bigg( x \mathbf I + \sum_{i=1}^m w_i B_i \bigg) \bigg|_{\genfrac{}{}{0pt}{2}{z_1=\ldots=z_m=0}{w_1=\ldots=w_m=0}}
\]
is at most 4.
\end{theorem}

\begin{proof}
Let where $Q$ be the polynomial given by \eqref{qx1}.
We will show a stronger conclusion that for any disjoint sets $S,S' \subset [m]$, the point $(4,0,\ldots,0)$ lies above the roots of the polynomial
\begin{equation}\label{qxx2}
p(x,(z_i)_{i\in [m] \setminus S'}, (w_i)_{i\in [m] \setminus (S \cup S')})=\prod_{i\in S} (1 - \tfrac 12\partial^2_{z_i\equiv w_i}) \prod_{i\in S'} (1- \partial_{z_i}\partial_{w_i})Q \bigg|_{z_i=w_i=0,\ i \in S'}.
\end{equation}
Indeed, this claim is easily shown by induction on the size of the set $S'$. By Lemma \ref{spd} the polynomial $p$ is  stable. If $S'=\emptyset$, this claim is just Lemma \ref{qx}. Next suppose that for some $k=0,\ldots,m-1$, the point $(4,0,\ldots,0)$ lies above the roots of every polynomial $p$ of the form \eqref{qxx2} corresponding to any set $S'$ of size $k$ and any set $S$ with $S \cap S'=\emptyset$.  
To show that the same holds by adding an extra element $i_0\in [m] \setminus S $ to the set $S'$, observe that by induction hypothesis the point $(4,0,\ldots,0)$ lies above the roots of 
\[
(1 - \tfrac 12 \partial^2_{z_{i_0} \equiv w_{i_0}})p \bigg|_{z_{i_0}=0}.
\]
By Lemma \ref{qz}, the same point $(4,0,\ldots,0)$ lies above the roots of 
\[
(1 - \partial_{z_{i_0}}\partial_{w_{i_0}})p \bigg|_{z_{i_0}=w_{i_0}=0}
=
\prod_{i\in S} (1 - \tfrac 12\partial^2_{z_i\equiv w_i}) \prod_{i\in S'\cup\{i_0\}} (1- \partial_{z_i}\partial_{w_i})Q \bigg|_{z_i=w_i=0,\ i \in S'\cup\{i_0\}}.
\]
This shows that claim. Taking $S=\emptyset$ and $S'=[m]$ yields the lemma.
\end{proof}

We are now ready to prove Theorem \ref{kls}.

\begin{proof}[Proof of Theorem \ref{kls}]
Replacing random variables $\xi_i$ by their mean zero counterpart $\xi_i - \E{}{\xi_i}$, we can assume without loss of generality that $E{}{\xi_i}=0$ for all $i\in [m]$. Indeed, this change does not affect the definition of $\sigma^2$ nor the conclusion \eqref{kls2}. Likewise, by rescaling random variables $\xi_i$ by the factor $1/\sigma$, we can also assume that $\sigma=1$.

Consider a random polynomial $f_{\xi_1,\ldots,\xi_m}$ as in Lemma \ref{muxx}. By Theorems \ref{maxint} and \ref{qint}, with positive probability
\[
\maxroot f_{\xi_1,\ldots,\xi_m} \le \maxroot \E{}{f_{\xi_1,\ldots,\xi_m}}.
\]
By Lemma \ref{muxx} and Theorem \ref{qxx} with $B_i=\var{\xi_i}^{1/2}A_i$ we have
\[
\maxroot \E{}{f_{\xi_1,\ldots,\xi_m}}
=\maxroot \mu_2[\var{\xi_1}^{1/2} A_1,\ldots, \var{\xi_m}^{1/2} A_m] \le 4.
\]
Applying Lemma \ref{norm} yields that for some outcome we have
\[
\bigg\| \sum_{i=1}^m \xi_i A_i \bigg\|
\le \maxroot f_{\xi_1,\ldots,\xi_m} \le 4. 
\]
This proves \eqref{kls2}.
\end{proof}

The positive definite assumption in Theorem \ref{kls} is not essential and can be relaxed by hermitian assumption with slightly worse constant.

\begin{corollary}\label{klsh}
Suppose that $\xi_1,\ldots, \xi_m$ are independent scalar random variables which take finitely many values.  Let $B_1, \ldots, B_m$ be $d \times d$ hermitian matrices. Define 
\begin{equation}\label{klsh1}
\sigma^2= \max \bigg( \max_{i=1,\ldots,m} \var{\xi_i} \tr(|B_i|)^2, \bigg\| \sum_{i=1}^m \var{\xi_i} \tr(|B_i|) |B_i| \bigg\| \bigg),
\end{equation}
where $|B_i|=(B_iB_i^*)^{1/2}$ denotes the absolute value of $B_i$.
Then, for some outcome
\begin{equation}\label{klsh2}
\bigg\| \sum_{i=1}^m (\xi_i-\E{}{\xi_i}) B_i \bigg\| \le 8 \sigma .
\end{equation}
\end{corollary}

\begin{proof}
Consider $2d\times 2d$ matrices $A_1, \ldots, A_m$ given in block diagonal form as
\[
A_i = \begin{bmatrix} (B_i)_+ & \\ & (B_i)_-
\end{bmatrix},
\]
where $B_+$ and $B_-$ denote positive and negative parts of a hermitian matrix $B$. Clearly, $\tr(A_i)=\tr(|B_i|)$ and
\[
A_i \le \begin{bmatrix} |B_i| & \\ & |B_i|
\end{bmatrix}.
\]
Hence, the constant $\sigma$ in \eqref{kls1} is dominated by that in \eqref{klsh1}.
By Theorem \ref{kls} there exists an outcome such that \eqref{kls2} holds. That is, for some outcome we have
\[
\bigg\| \sum_{i=1}^m (\xi_i-\E{}{\xi_i}) (B_i)_+ \bigg\|\le 4\sigma
\quad\text{and}\quad
\bigg\| \sum_{i=1}^m (\xi_i-\E{}{\xi_i}) (B_i)_- \bigg\| \le 4 \sigma.
\]
Since $B_i=(B_i)_+-(B_i)_-$, the above yields \eqref{klsh2}.
\end{proof}

\section{Akemann-Weaver conjecture}\label{S4}

In this section we prove the Akemann-Weaver conjecture which was posed in \cite{AW}. That is, we establish the approximate Lyapunov's theorem for trace class operators generalizing rank one Lyapunov-type result of Akemann and Weaver \cite[Theorem 2.4]{AW} by simultaneously improving the bound from $O(\epsilon^{1/8})$ to $O(\epsilon^{1/2})$. The bound improvement for rank one operators was first shown by Kyng, Luh, and Song \cite[Corollary 1.8]{KLS} with explicit small constant that was further improved in \cite{XXZ1}. We show that the same bound holds for higher ranks positive definite matrices as well as positive trace class operators thus resolving the Akemann-Weaver conjecture in affirmative.

\begin{theorem}
\label{aw}
Let $I$ be countable and $\epsilon>0$. Suppose that $\{T_i\}_{i\in I}$ is a family of positive trace class operators in a separable Hilbert space $\mathcal{H}$ such that
\begin{equation}\label{aws}
\sum_{i\in I} T_i \le \bI \quad\text{and}\quad \tr(T_i)\le \epsilon \qquad\text{for all }i \in I.
\end{equation}
Suppose that $0\le t_i \le 1$ for all $i\in I$. Then, there exists a subset of indices $I_0 \subset I$ such that
\begin{equation}\label{aw0}
\bigg\| \sum_{i\in I_0} T_i  - \sum_{i\in I} t_i T_i  \bigg\| \le 2 \epsilon^{1/2}.
\end{equation}
\end{theorem}

Before proving Theorem \ref{aw} we explore some consequences of a higher rank extension of Marcus-Spielman-Srivastava result, Theorem \ref{thmp}.

\begin{lemma}\label{limp}
Let $\epsilon>0$. Suppose $A_1, \ldots, A_m$ are $d \times d$ positive semidefinite matrices satisfying 
\begin{equation}
A:=\sum_{i=1}^m A_i \le \bI 
\qquad\text{and}\qquad
\tr(A_i) \le \epsilon \quad\text{for all i}.
\end{equation}
Let $t_k>0$, $k=1,\ldots, r$, where $r\in \N$, be such that $\sum_{k=1}^r t_k=1$. Then, there exists a partition $\{I_1,\ldots, I_r\}$ of $[m]$ such that such that
\begin{equation}\label{limp0}
 \sum_{i\in I_k} A_i  \le t_k (A  + (2\sqrt{r\epsilon} + r\epsilon) \mathbf I).
\end{equation}
\end{lemma}

\begin{proof}
Since $\mathbf I-A$ is positive semidefinite, we can find a collection $B_1,\ldots,B_{m'}$ of positive semidefinite matrices  such that
\[
\mathbf I-A = \sum_{i=1}^{m'} B_i \qquad\text{and}\qquad \tr(B_i) \le \epsilon \text{ for }i \in [m'].
\]
Indeed, we can can construct such rank 1 matrices by choosing vectors $\{u_i\}_{i=1}^{m'}$ to be appropriately scaled eigenvectors of $\mathbf I-A$ so that $||u_i||^2 \le \epsilon$, and letting $B_i=u_i u_i^*$.

Let $X_1,\ldots,X_m$ be jointly independent $dr \times dr$ random matrices such that each matrix $X_i$ is block diagonal 
taking values
\[
\begin{bmatrix} t_1^{-1} A_i  &  & & \\  & \mathbf 0_d & & \\ &&\ddots & \\ &&& \mathbf 0_d \end{bmatrix}, \ldots,
\begin{bmatrix}    \mathbf 0_d & &&\\ &\ddots&&  \\ &&\mathbf 0_d& \\ &&&t_r^{-1} A_i\end{bmatrix}
\]
with probability $t_k$, $k=1,\ldots,r$, respectively. Here, $\mathbf 0_d$ denotes $d \times d$ zero matrix. Let $Y_1,\ldots,Y_{m'}$ be $dr \times dr$ deterministic block diagonal matrices taking one value
\[
Y_i=\begin{bmatrix} B_i & & \\
& \ddots & \\
& &  B_i
\end{bmatrix}
\text{ with probability $1$}.
\]
Then,
\[
\begin{aligned}
\sum_{i=1}^m \E{}{X_i} + \sum_{i=1}^{m'} \E{}{Y_i}
& =  \begin{bmatrix} \sum_{i=1}^m A_i &   &  \\   & \ddots &  \\  &  & \sum_{i=1}^m A_i \end{bmatrix} + 
\begin{bmatrix} \sum_{i=1}^{m'} B_i &   &  \\   & \ddots &  \\  &  & \sum_{i=1}^{m'} B_i \end{bmatrix} 
=\mathbf I_{dr},
\\
\E{}{\tr(X_i) }  & = \sum_{k=1}^r t_k (t_k)^{-1} \tr(A_i) \le r\epsilon \qquad i=1,\ldots,m,\\
\E{}{\tr(Y_i) }  & = \sum_{k=1}^r \tr(B_i) \le r\epsilon \qquad i=1,\ldots,m'.
\end{aligned}
\]
Hence, we can apply Theorem \ref{thmp} for random matrices $X_1,\ldots, X_m,Y_1,\ldots,Y_{m'}$ with $r\epsilon$ in place of $\epsilon$. Choose an outcome for which the bound in \eqref{thmp2} happens. For this outcome define
\[
I_k = \{i \in [m]: X_i \text{ has non-zero in $k^{\rm th}$ entry} \}, \qquad  k=1,\ldots,r.
\]
Thus, the block diagonal matrix
\[
\sum_{i=1}^m X_i + \sum_{i=1}^{m'} Y_i
= \begin{bmatrix}  t_1^{-1} \sum_{i\in I_1} A_i + \mathbf I - A &   &  \\   & \ddots &  \\  &  & t_r^{-1} \sum_{i\in I_r} A_i + \mathbf I - A \end{bmatrix}
\]
has norm bounded by $(1 + \sqrt{r \epsilon})^2$. Consequently, for any $k=1,\ldots, r$,
\[
\sum_{i\in I_k} A_i \le -t_k(\mathbf I - A) + t_k(1+2\sqrt{r\epsilon} + r\epsilon) \mathbf I = t_k A + t_k(2\sqrt{r\epsilon} + r\epsilon) \mathbf I.
\]
This proves \eqref{limp0}.
\end{proof}

Lemma \ref{limp} implies a higher rank of Weaver's $KS_r$ conjecture, which was conjectured in \cite{AW}. Corollary \ref{MSS} in the case of rank one matrices was shown by Marcus-Spielman-Srivastava in their solution of the Kadison-Singer problem \cite[Corollary 1.3]{MSS}, see also \cite[Lemma 2.1]{AW}.

\begin{corollary}[Higher rank $KS_r$ conjecture] \label{MSS}
Let $\epsilon>0$. Suppose $A_1, \ldots, A_m$ are $d \times d$ positive semidefinite matrices satisfying \eqref{mixed1}, i.e.,
\begin{equation}\label{ms1}
\sum_{i=1}^m A_i \le \bI 
\qquad\text{and}\qquad
\tr(A_i) \le \epsilon \quad\text{for all i}.
\end{equation}
Let $t_k>0$, $k=1,\ldots, r$, where $r\in \N$, be such that $\sum_{k=1}^r t_k=1$. Then, there exists a partition $\{I_1,\ldots, I_r\}$ of $[m]$ such that 
\begin{equation}\label{ms2}
\bigg\| \sum_{i\in  I_k} A_i \bigg\| \le t_k \left(1 + \sqrt{r \epsilon}\right)^2 \qquad\text{for all } k=1,\ldots,r.
\end{equation}
\end{corollary}

In the special case $r=2$ and $t_1=t_2=1/2$, Corollary \ref{MSS} implies that we can partition any set of positive semidefinite matrices with small traces that sum up to the identity matrix, into two subsets each of which sums up to approximately half the identity matrix. That is, under the assumption 
\begin{equation}\label{ms3}
\sum_{i=1}^m A_i = \bI 
\qquad\text{and}\qquad
\tr(A_i) \le \epsilon \quad\text{for all i},
\end{equation}
there exists a partition $\{I_1,I_2\} \subset [m]$ such that
\[
\bigg\| \sum_{i\in I_k} A_i - \tfrac12 \mathbf I \bigg\| = O(\sqrt{\epsilon}) \qquad \text{as }\epsilon \to 0, \ k=1,2.
\]

Next we extend Lemma \ref{limp} to an infinite dimensional setting. In the proof of Theorem \ref{imp} we control only upper bounds in a transition from finite to infinite collections of operators and deduce the lower bounds only at the very end (as a consequence of the upper bounds).


\begin{theorem}\label{imp}
Let $I$ be countable and $\epsilon>0$. Suppose that $\{T_i\}_{i\in I}$ is a family of positive trace class operators in a separable Hilbert space $\mathcal{H}$ satisfying 
\begin{equation}\label{aws0}
T:=\sum_{i\in I} T_i \le \bI \quad\text{and}\quad \tr(T_i)\le \epsilon \qquad\text{for all }i \in I.
\end{equation}
Let $t_k>0$, $k=1,\ldots, r$, where $r\in \N$, be such that $\sum_{k=1}^r t_k=1$. Then, there exists a partition $\{I_1,\ldots, I_r\}$ of $I$ such that 
\begin{equation}\label{imp0}
\bigg\| \sum_{i\in  I_k} T_i - t_k T  \bigg\|  
\le 2\sqrt{r\epsilon} +r\epsilon \qquad\text{for all } k=1,\ldots,r.
\end{equation}
\end{theorem}

\begin{proof}
First, we consider the case when $I$ is finite. For any $i\in I$, choose a sequence of positive finite rank operators $\{T^{(n)}_i\}$ such that
\begin{equation}\label{iaw3}
0 \le T_i^{(n)} \le T_i \qquad\text{and}\qquad \lim_{n\to\infty} || T_i - T^{(n)}_i || =0.
\end{equation}
Take any $n\in \N$. Since operators $\{T^{(n)}_i\}_{i\in I}$ act non-trivially on some finite dimensional subspace $\mathcal K \subset \mathcal H$ we can identify them with positive semidefinite matrices on $\C^d$, $d =\dim \mathcal K$.
By Lemma \ref{limp} applied for a finite collection $\{T^{(n)}_i\}_{i\in I}$, we obtain a partition $\{I_1^n, \ldots, I_r^n\}$ of $I$ such that
\begin{equation}\label{iaw4}
\sum_{i\in  I_k^n} T^{(n)}_i  \le t_k \bigg( \sum_{i\in I} T^{(n)}_i + (2\sqrt{r\epsilon} + r\epsilon) \mathbf I\bigg)
\le t_k ( T + (2\sqrt{r\epsilon} + r\epsilon) \mathbf I)
\qquad\text{for all } k=1,\ldots,r.
\end{equation}
By pigeonhole principle, for infinitely many $n\in\N$, $\{I_1^n, \ldots, I_r^n\}$ is some fixed partition $\{I_1,\ldots,I_r\}$ of $I$. Taking limit in \eqref{iaw4} as $n\to\infty$ yields 
\begin{equation}\label{iaw2}
\sum_{i\in  I_k} T_i  \le t_k ( T + (2\sqrt{r\epsilon} + r\epsilon) \mathbf I)
\qquad\text{for all } k=1,\ldots,r.
\end{equation}

Now suppose that $I$ is countable. We may assume $I=\N$. For any $n\in \N$, we apply the above to the family $\{T_i\}_{i\in [n]}$. This yields  a partition $\{I_1^n, \ldots, I_r^n\}$ of $[n]$  such that \eqref{iaw2} holds with $I^n_k$ in place of $I_k$, $k=1,\ldots,r$. To show the existence of a global partition of $\{I_1,\ldots,I_r\}$ of $\N$ satisfying \eqref{iaw2}, it suffices to apply the pinball principle, see \cite[Lemma 2.7]{MB}. Alternatively, we can use a compactness arguments as in the proof of \cite[Theorem 3.1]{AW}. Either way we obtain \eqref{iaw2}.

Finally, we deduce the estimate \eqref{imp0} from \eqref{iaw2}. Fix $k_0=1,\ldots,r$. Adding up \eqref{iaw2} over $k\ne k_0$ yields
\[
T - \sum_{i\in  I_{k_0}} T_i = \sum_{ \genfrac{}{}{0pt}{}{k=1}{k \ne k_0}}^r \sum_{i\in I_k} T_i \le (1-t_{k_0}) ( T + (2\sqrt{r\epsilon} + r\epsilon) \mathbf I).
\]
By rearranging we obtain
\begin{equation}\label{iaw6}
-(1-t_{k_0})  (2\sqrt{r\epsilon} + r\epsilon) \mathbf I \le \sum_{i\in  I_{k_0}} T_i - t_{k_0}T.
\end{equation}
Combining \eqref{iaw2} with \eqref{iaw6} yields 
\[
\bigg\|\sum_{i\in  I_{k_0}} T_i - t_{k_0}T \bigg\| \le \max(t_{k_0},1-t_{k_0}) (2\sqrt{r\epsilon} + r\epsilon).
\]
In particular, we have \eqref{imp0}.
\end{proof}

As an immediate corollary of Theorem \ref{imp} we obtain the special case of Theorem \ref{aw}, which simultaneously gives improved bounds and extension to the higher rank one case of \cite[Lemma 2.3]{AW}, see also \cite[Corollary 2.3]{FY}. 

\begin{corollary}\label{iawc} Suppose that $\{T_i\}_{i\in I}$ is a family of positive trace class operators in a separable Hilbert space $\mathcal{H}$ satisfying \eqref{aws0} for some $\epsilon>0$. Then, for any $0<t<1$, there exists a subset $I_ 0 \subset I$ such that
\[
\bigg\| \sum_{i \in I_0} T_i - t T \bigg\| \le  2\sqrt{2\epsilon} +2\epsilon =O(\sqrt{\epsilon}) \qquad\text{as }\epsilon \to 0.
\]
\end{corollary}

Using the approach of Akemann and Weaver \cite[Theorem 2.4]{AW} and Corollary  \ref{iawc} we can show Theorem \ref{aw} with a suboptimal bound  $O(\epsilon^{1/4})$.

\begin{proof}[Proof of Theorem \ref{aw} with non-optimal bound]
Take $n=\lfloor \epsilon^{-1/4} \rfloor$ and define subsets 
\[
I_k = \{ i \in I: (k-1)/n< t_i \le k/n \}, \qquad k=1,\ldots,n.
\]
Then, we apply Corollary \ref{iawc} for each family $\{T_i\}_{i\in I_k}$ for $t=k/n$ to find subsets $I_k' \subset I_k$ such that
\[
\bigg\| \sum_{i \in I'_k} T_i - \frac{k}{n}  \sum_{i\in I_k} T_i \bigg\| = O(\epsilon^{1/2}).
\]
Taking $I_0 = \bigcup_{k=1}^n I'_k$, we have
\begin{align*}
\bigg\| \sum_{i \in I_0} T_i -  \sum_{i\in I} t_i T_i \bigg\| 
& \le 
 \bigg\| \sum_{k=1}^n  \bigg(\sum_{i \in I'_k} T_i - \frac{k}{n}  \sum_{i\in I_k} T_i \bigg) \bigg\| +  \bigg\| \sum_{k=1}^n  \sum_{i\in I_k} (k/n-t_i)T_i \bigg\|
  \\
& \le    \sum_{k=1}^n  \bigg\| \sum_{i \in I'_k} T_i - \frac{k}{n}  \sum_{i\in I_k} T_i \bigg\| + \frac1n \bigg\|\sum_{i\in I} T_i \bigg\| 
\\
& \le n O(\epsilon^{1/2})  + O(\epsilon^{1/4}) =  O(\epsilon^{1/4}). \qedhere
\end{align*}
\end{proof}

To show Theorem \ref{aw} in full strength we shall employ Theorem \ref{kls} instead of Theorem \ref{thmp}.

\begin{proof}[Proof of Theorem \ref{aw}]
We can assume that $\epsilon<1$ since otherwise \eqref{aw0} holds trivially.
As in the proof of Theorem \ref{imp}, we first consider the case when $I$ is finite. For any $i\in I$, choose a sequence of positive finite rank operators $\{T^{(n)}_i\}_{n\in \N}$ such that \eqref{iaw3} holds.
For fixed $n\in \N$ operators $T^{(n)}_i$, $i \in I$, act non-trivially on some finite dimensional subspace $\mathcal K \subset \mathcal H$. Hence, we can identify them with positive semidefinite $d\times d$ matrices, where $d =\dim \mathcal K$.

Let $\xi_i$, $i\in I$, be independent random variables such that $\xi_i$ takes values $0$ and $1$ with probabilities $1-t_i$ and $t_i$, respectively. Since $\var{\xi_i} = t_i(1-t_i) \le 1/4$, by \eqref{aws} and \eqref{iaw3} we have
\[
\begin{aligned}
\sigma^2 &= \max \bigg( \max_{i\in I } \var{\xi_i} \tr(T_i^{(n)})^2, \bigg\| \sum_{i=1}^m \var{\xi_i} \tr(T_i^{(n)}) T_i^{(n)} \bigg\| \bigg)
\\
& \le \max \bigg( \frac 14 \epsilon^2, \frac 14 \epsilon  \bigg\| \sum_{i\in I} T_i \bigg\| \bigg) \le  \frac 14 \epsilon.
\end{aligned}
\]
By Theorem \ref{kls} applied for a finite collection $\{T^{(n)}_i\}_{i\in I}$, there exists a subset $J_n \subset I$ such that
\[
\bigg\| \sum_{i\in J_n} T_i^{(n)}  - \sum_{i\in I} t_i T_i^{(n)}  \bigg\| \le 2 \epsilon^{1/2}.
\]
By pigeonhole principle, for infinitely many $n\in\N$, $J_n$ is some fixed subset $I_0$ of $I$. Taking limit in \eqref{iaw4} as $n\to\infty$ yields 
\begin{equation}\label{kls5}
\bigg\| \sum_{i\in I_0} T_i  - \sum_{i\in I} t_i T_i  \bigg\| \le 2 \epsilon^{1/2}.
\end{equation}
This can be equivalently written as
\begin{equation}\label{kls6}
\sum_{i\in I_0} T_i \le \sum_{i\in I} t_i T_i + 2 \epsilon^{1/2} \mathbf I 
\quad\text{and}\quad
\sum_{i\in I \setminus I_0} T_i \le \sum_{i\in I} (1-t_i) T_i + 2 \epsilon^{1/2} \mathbf I.
\end{equation}

Now suppose that $I$ is countable. We may assume $I=\N$. For any $n\in \N$, we apply the above to the family $\{T_i\}_{i\in [n]}$. This yields  a partition $\{I_n, [n] \setminus I_n\}$ of $[n]$  such that \eqref{kls6} holds with $I_0$ and $I$ replaced by $I_n$ and $[n]$, respectively. To show the existence of a global partition of $\{I_0,I \setminus I_0\}$ of $I$ satisfying \eqref{kls6}, it suffices to apply the pinball principle, see \cite[Lemma 2.7]{MB} or a compactness arguments as in the proof of \cite[Theorem 3.1]{AW}. 
This yields \eqref{kls5}.
\end{proof}

\bibliographystyle{amsplain}

\end{document}